\documentclass{amsart}
\usepackage{amsmath,amssymb,amscd,amsthm}
\usepackage[all]{xy}
\usepackage[colorlinks,linkcolor = blue,citecolor = blue]{hyperref}
\usepackage{color}
\usepackage{array}
\makeatletter
 \def\Spec{\mathop{\operator@font Spec}\nolimits}
 \def\Aut{\mathop{\operator@font Aut}\nolimits}
 \def\Out{\mathop{\operator@font Out}\nolimits}
 \def\Image{\mathop{\operator@font Im}\nolimits}
 \def\Gal{\mathop{\operator@font Gal}\nolimits}
 \def\Isom{\mathop{\operator@font Isom}\nolimits}
 \def\OutIsom{\mathop{\operator@font OutIsom}\nolimits}
 \def\ker{\mathop{\operator@font ker}\nolimits}
 \def\cdim{\mathop{\operator@font cd}\nolimits}
 \def\id{\mathop{\operator@font id}\nolimits}
\makeatother

\theoremstyle{plain}
\newtheorem{thm}{Theorem}[section]
\newtheorem{prop}[thm]{Proposition}
\newtheorem{cor}[thm]{Corollary}
\newtheorem{lem}[thm]{Lemma}
\newtheorem{introthm}{Theorem}

\newtheorem{introcor}[introthm]{Corollary}

\theoremstyle{definition}
\newtheorem{defi}[thm]{Definition}

\theoremstyle{remark}
\newtheorem{rem}{Remark}[thm]
\newtheorem{e.g.}[thm]{Example}

\allowdisplaybreaks[1]

\begin{document}

\title[Families preserving isomorphisms in anabelian geometry]{Families preserving isomorphisms via techniques in anabelian geometry: with an application to a generalized Neukirch-Uchida theorem}

\author{Arata Minamide}
\address[Arata Minamide]{ZEN Mathematics Center, ZEN University, Zushi 249-0007, Japan; Research Institute for Mathematical Sciences, Kyoto University, Kyoto 606-8502, Japan}
\email{arata\_minamide@zen.ac.jp; minamide@kurims.kyoto-u.ac.jp}

\author{Koichiro Sawada}
\address[Koichiro Sawada]{Research Institute for Mathematical Sciences, Kyoto University, Kyoto 606-8502, Japan}
\email{sawada@kurims.kyoto-u.ac.jp}

\author{Shota Tsujimura}
\address[Shota Tsujimura]{Research Institute for Mathematical Sciences, Kyoto University, Kyoto 606-8502, Japan}
\email{stsuji@kurims.kyoto-u.ac.jp}


\maketitle

\begin{abstract}
Families preserving automorphisms of profinite groups are automorphisms that preserve each of the conjugacy classes of pro-cyclic subgroups. This notion appears in the context of verification of the property that every normal automorphism [i.e., an automorphism that preserves each of the normal closed subgroups] of a certain profinite group such as a nonabelian free profinite group or the absolute Galois group of a $p$-adic local field is an inner automorphism, which was proved by Jarden/Jarden-Ritter.
In the present paper, we revisit the notion of a families preserving automorphism from the viewpoint of anabelian geometry. We introduce a natural generalized version of this notion for isomorphisms between \textit{general} closed subgroups of a profinite group, which we shall refer to as {\it families preserving isomorphisms} in the profinite group. With regard to this generalized version, we prove that a large class of profinite groups satisfies a property that families preserving isomorphisms of certain closed subgroups of the profinite groups arise from inner automorphisms of them. For instance, the class includes the absolute Galois groups of Henselian discrete valuation fields of positive residue characteristic or Hilbertian fields. Moreover, as an application of this property, together with delicate considerations on pro-cyclic subgroups of the absolute Galois groups of infinite algebraic extension fields of the field of rational numbers $\mathbb{Q}$, we obtain a generalized version of the Neukirch-Uchida theorem for {\it $l$-quasi-number fields}, where an $l$-quasi-number field is defined to be an algebraic extension field $K$ of $\mathbb{Q}$ whose Galois closure $L$ over $\mathbb{Q}$ satisfies the property that the extension degree $[L : \mathbb{Q}]$ is divisible by a prime number $l$ only finitely many times. Surprisingly, this result includes the Neukirch-Uchida-type result for the class of subfields of the maximal pro-prime-to-$l$ extension fields of arbitrary number fields.
\end{abstract}

\renewcommand{\labelenumi}{(\roman{enumi})}
\renewcommand{\labelenumii}{(\arabic{enumii})}

\tableofcontents

\setcounter{tocdepth}{0}
\section*{Introduction}

Throughout the present paper, isomorphisms of topological groups are always assumed to be continuous. For each profinite group $G$ and each automorphism $\sigma$ of $G$, we shall say that
\begin{itemize}
\item
$\sigma$ is {\it families preserving} if the image of any pro-cyclic subgroup $I$ via $\sigma$ is conjugate to $I$ in $G$;
\item
$\sigma$ is {\it normal} if $\sigma$ preserves each of the normal closed subgroups of $G$.
\end{itemize}
In particular, it is immediate that families preserving automorphisms of profinite groups are normal. Recall that these notions appear in the context of the Neukirch-Uchida theorem that bridge the works of Neukirch [cf.\ \cite{N1}, \cite{N2}] and the work of Uchida [cf.\ \cite{U1}].  Here, we note that there are contributions of Iwasawa and Ikeda in this context. For instance, Ikeda also showed that every automorphism of the absolute Galois group of the field of rational numbers $\mathbb{Q}$ is an inner automorphism at the almost same time [cf.\ \cite{I}]. After that, apart from the arithmetic setting, Jarden proved that this property also holds for nonabelian free profinite groups [cf.\ \cite{J}]. Later, Jarden's result was generalized substantially by Jarden-Ritter in the following way. Let $p$ be a prime number. Then we shall say that a profinite group $G$ is {\it pseudo-$p$-free} if, for each normal open subgroup $N \subseteq G$, there exists a closed subgroup $M \subseteq N$ that is normal in $G$ such that $N/M$ is abelian, and $N/M$ admits a $G/N$-invariant pro-$p$ closed subgroup isomorphic to the group ring $\mathbb{Z}_p[G/N]$. For this technically defined class of profinite groups, they proved that [cf.\ \cite{JR}, Theorem 1]:
\begin{thm}[Jarden-Ritter]\label{JR-1}
Every normal automorphism of a pseudo-$p$-free profinite group is an inner automorphism.
\end{thm}
As concrete nontrivial consequences, they proved that [cf.\ \cite{JR}, Theorems A, B]:
\begin{thm}[Jarden-Ritter]\label{JR-2}
Every normal automorphism of the absolute Galois group of a $p$-adic local field is an inner automorphism.
\end{thm}
\begin{thm}[Jarden-Ritter]\label{JR-3}
Let $G$ be a profinite group; $d$, $e$ nonnegative integers such that $e \geq d+2$; $\mathcal{C}$ a full-formation. Suppose that $G$ is a pro-$\mathcal{C}$ group presented by $e$ generators and $d$ word-relations. Then every normal automorphism of $G$ is an inner automorphism.
\end{thm}
Their proof of Theorem \ref{JR-1} was divided into two steps: first showing that every normal automorphism is families preserving, and then showing that every families preserving automorphism is an inner automorphism.
In the present paper, we focus on a stronger version of the latter property. By applying group-theoretic techniques in anabelian geometry, we prove that many highly nonabelian profinite groups satisfy the stronger property. For our purpose, we generalize the notion of a families preserving automorphism as follows: 
\begin{quotation}
\noindent
For an isomorphism between closed subgroups of a profinite group $G$, if the
image of any pro-cyclic subgroup $I$ via this isomorphism is conjugate to $I$ in $G$, then we
shall say that this isomorphism is {\it families preserving} in $G$ [cf.\ Definition \ref{snfp}].
\end{quotation}
Then, as a preliminary technical result that corresponds to Theorem \ref{JR-1} in our context, we prove the following result, which may be regarded as our first main result [cf.\ a special case of Theorem \ref{autfree}; Remark \ref{snfprem}, (ii)]:
\begin{introthm}\label{introthmA}
Let $G$ be a profinite group; $F\subseteq G$ a nonabelian free pro-$p$ normal closed subgroup; $\alpha\in\Aut(F)$ an automorphism of $F$ that is families preserving in $G$. Then $\alpha$ is induced by an inner automorphism of $G$.
\end{introthm}
It is an interesting point of Theorem \ref{introthmA} that the property that a families preserving automorphism of a nonabelian free pro-$p$ group in a container arises from an inner automorphism of the container is independent of the choice of a container.
With regard to the proof of Theorem \ref{introthmA}, it would be important to note that the discussion is based on group-theoretic techniques in anabelian geometry. On the other hand, the proof of Theorem \ref{JR-1} is based on the discussions in abelianized setting. At the time of writing of the present paper, the precise relation between the notion of a pseudo-$p$-free group that appears in Theorem \ref{JR-1} and the property that $F$ is a nonabelian free pro-$p$ normal closed subgroup that appears in Theorem A is not clear to the authors. 

Next, in order to discuss applications of [a strengthened version of] Theorem \ref{introthmA}, we introduce some notations and definitions. For each field $F$, write $G_F$ for the absolute Galois group of $F$ with respect to a fixed separable closure of $F$. 
For each full-formation $\mathcal{C}$ and each quotient $Q$ of a profinite group $G$ [in the category of profinite groups]:
\begin{enumerate}
\item We shall write $G^{\mathcal{C}}$ for the maximal pro-$\mathcal{C}$ quotient of $G$. If $\mathcal{C}$ is the family of all $p$-groups, then we shall also write $G^p\overset{\mathrm{def}}{=}G^{\mathcal{C}}$ [cf.\ Definition \ref{almostquot}, (i)].
\item If there exists a normal open subgroup $N\subseteq G$ of $G$ such that the kernel of the surjection $G\twoheadrightarrow Q$ coincides with the kernel of the natural surjection $N\twoheadrightarrow N^{\mathcal{C}}$ (respectively, $N\twoheadrightarrow N^p$),
then we shall say that $Q$ is the \textit{almost pro-$\mathcal{C}$-maximal quotient} (respectively, \textit{almost pro-$p$-maximal quotient}) of $G$ [cf.\ Definition \ref{almostquot}, (iii)].
\end{enumerate}
Then, as applications of Theorem \ref{introthmA} [more precisely, Theorem \ref{innisom}, which may be regarded as a strengthened version of Theorem \ref{introthmA} under a certain assumption on $G$], together with various works on anabelian group-theoretic properties due to the authors of the present paper; highly nontrivial results in field arithmetic, we obtain the following concrete results concerning families preserving isomorphisms [cf.\ special cases of Theorems \ref{inngal}, \ref{innanab1}]:
\begin{introthm}\label{introthmB}
Let $K$ be a Hilbertian field or a Henselian discrete valuation field of residue characteristic $p$; $H,H'\subseteq G_K$ closed subgroups. Suppose that $H,H'$ admit nontrivial closed subgroups that are normal in $G_K$. Then every families preserving isomorphism $H\overset{\sim}{\to}H'$ in $G_K$ is induced by an inner automorphism of $G_K$. Moreover, in the case where $K$ is a number field or a Henselian discrete valuation field of residue characteristic $p$, for any full-formation $\mathcal{C}$ such that $\mathbb{Z}/p\mathbb{Z}$ is a $\mathcal{C}$-group, the same property also holds for the almost pro-$\mathcal{C}$-maximal quotients of $G_K$.
\end{introthm}
\begin{introthm}\label{introthmC}
Let $G$ be a profinite group; $\mathcal{C}$ a nontrivial full-formation. Suppose that $G$ is isomorphic to one of the following:
\renewcommand{\labelenumi}{(\alph{enumi})}
\begin{enumerate}
\item an almost pro-$\mathcal{C}$-maximal quotient of a free profinite group of [possibly infinite] rank $\ge 2$;
\item an almost pro-$\mathcal{C}$ surface group;
\item a pro-$p$ Demu\v{s}kin group of rank $\ge 3$.
\end{enumerate}
Let $H,H'\subseteq G$ be closed subgroups. Suppose that $H,H'$ admit nontrivial closed subgroups that are normal in $G$. Then every families preserving isomorphism $H\overset{\sim}{\to}H'$ in $G$ is induced by an inner automorphism of $G$.
\end{introthm}
In fact, in order to apply Theorem \ref{introthmA} to obtain the above results, we need to consider certain ``almost pro-$p$ quotients" to construct nontrivial free pro-$p$ normal closed subgroups in them. However, this is by no means trivial in the case of Theorem \ref{introthmB}. In the case where $K$ is a Henselian discrete valuation field of residue characteristic $p$ in Theorem \ref{introthmB}, one may use the almost pro-$p$-maximal quotients. Even in this case, we resort to the first and third authors' previous work based on Artin-Schreier theory in characteristic $0$ due to MacKenzie-Whaples, together with the highly nontrivial theory of fields of norms due to Fontaine-Wintenberger [cf.\ \cite{W}; \cite{FV}, Chapter III, \S 5]. On the other hand, in the case where $K$ is a Hilbertian field, at the time of writing of the present paper, it is not clear to the authors that $G_{K}^p$ admits a nontrivial free pro-$p$ normal closed subgroup. To deal with this situation, we introduce the notion of an {\it almost symmetric pro-$p$-maximal quotient} related to finite quotients isomorphic to symmetric groups [cf.\ Definition \ref{symquot}]. In these quotients, in light of highly nontrivial properties of Hilbertian fields, one may prove the existence of desired nontrivial free pro-$p$ normal closed subgroups. Here, we note that it may be difficult to obtain results for large classes of closed subgroups as in Theorems \ref{introthmB}, \ref{introthmC}, via the discussions in abelianized settings as in \cite{JR}. This point may be regarded as one of the main differences between \cite{JR} and our work. Note also that the class of profinite groups that appears in Theorem \ref{introthmC} is smaller than that of Theorem \ref{JR-3}. If there exists a way to construct a nontrivial free pro-$p$ normal closed subgroup in suitable quotients of profinite groups in Theorem \ref{JR-3}, then one may cover the larger class of profinite groups in Theorem \ref{introthmC}. However, at the time of writing of the present paper, the authors have no clue to obtain such a generalization. For another note, in the main text, we also discuss families preserving isomorphisms for group extensions of profinite groups that appear in Theorems \ref{introthmB}, \ref{introthmC}.

In any event, with regard to families preserving isomorphism, one may cover many profinite groups that appear in anabelian geometry. However, at the time of writing of the present paper, the authors do not know whether or not every normal automorphism is an inner automorphism in the situation of Theorem \ref{introthmB}. While pursuing the potential generalization of the above results to the situation of normal automorphisms as in the works of Jarden/Jarden-Ritter is undoubtedly interesting, it is our observation that the notion of a families preserving automorphism/isomorphism is more important than the notion of a normal automorphism in terms of the point that the former may be regarded as a reminiscence of the notion of a point-theoretic automorphism/isomorphism, which appears frequently in anabelian geometry. This point will be discussed in the final portion of the present Introduction.

Next, we discuss an application to a generalized Neukirch-Uchida theorem. For each pair of fields $F_1$, $F_2$, write 
\begin{equation*}
\Isom(F_1, F_2)
\end{equation*}
for the set of isomorphisms $F_1\overset{\sim}{\to}F_2$ of fields. For each pair of profinite groups $G_1$, $G_2$, write 
\begin{equation*}
\OutIsom(G_1, G_2)
\end{equation*}
for the set of outer isomorphisms $G_1\overset{\sim}{\to}G_2$ of profinite groups. Then the classical Neukirch-Uchida theorem is as follows [cf.\ \cite{U1}, Theorem; \cite{U2}, Theorem]:
\begin{thm}[Neukirch-Uchida]\label{NU}
Let $K_1$, $K_2$ be global fields. Then the natural map
\begin{equation*}
\Isom(K_2,K_1)\to\OutIsom(G_{K_1},G_{K_2})
\end{equation*}
is bijective. In particular, $K_1$ and $K_2$ are isomorphic if and only if $G_{K_1}$ and $G_{K_2}$ are isomorphic [as profinite groups].
\end{thm}
Note that various generalizations of Theorem \ref{NU} with respect to replacing the absolute Galois groups by smaller quotients have been investigated intensively [cf.\ e.g., \cite{KS}, \cite{O2}, \cite{ST}, \cite{Sh1}, \cite{Sh2}, \cite{Sh3}, \cite{U3}]. On the other hand, the direction of extending the class of fields has not been sufficiently developed. It appears to the authors that this direction is also important from the both viewpoints
\begin{itemize}
\item
exploration of the limit of anabelian phenomena;
\item
existence of generalized class field theory for certain infinite algebraic extension fields of global fields [cf.\ e.g., \cite{Ka}, \cite{RW}]. 
\end{itemize}
In light of historical developments of class field theory and anabelian geometry, together with the close relationship between them, it is natural to ask whether or not there exist anabelian phenomena for fields with sufficiently well-behaved class field theory. With regard to the positive characteristic situation, global class field theory for the function fields of curves over quasi-finite fields that are algebraic over the prime fields was established by Rim-Whaples in \cite{RW}. In this case, recently, the third author proved the corresponding Neukirch-Uchida-type result [cf.\ \cite{T}, Corollary F, (ii)]. On the other hand, in the case of characteristic $0$, after the works of Herbrand, Krull, and Moriya, from the viewpoint of id\`ele theory, Kawada established class field theory for a large class of algebraic extension fields of $\mathbb{Q}$, which is defined by a certain type of finiteness of extension degrees over $\mathbb{Q}$ in the sense of supernatural numbers [cf.\ \cite{Ka}]. In light of this development, in the present paper, as an application of results concerning families preserving isomorphisms, together with delicate considerations on pro-cyclic subgroups of the absolute Galois groups of infinite algebraic extension fields of $\mathbb{Q}$, we prove a Neukirch-Uchida-type result for a certain large class of algebraic extension fields of $\mathbb{Q}$. Here, we note that our work is not the first attempt to prove a Neukirch-Uchida-type result for infinite extension fields of $\mathbb{Q}$ [see the Ozaki's work \cite{O1} --- cf.\ Remark \ref{lqnfinnrem}]. 

Let $K$ be an algebraic extension field of $\mathbb{Q}$. Then, for each prime number $l$, we shall say that $K$ is an \textit{$l$-quasi-number field} if the Galois closure $L$ of $K$ over $\mathbb{Q}$ satisfies the property that the extension degree $[L : \mathbb{Q}]$ [in the sense of super natural numbers] is divisible by $l$ only finitely many times [cf.\ Definition \ref{qnf}]. Note that many infinite algebraic extension fields of $\mathbb{Q}$ are $l$-quasi-number fields for some prime number $l$. For instance, any subfield of the maximal pro-prime-to-$l$ extension field of a number field is an $l$-quasi-number field [cf.\ Proposition \ref{qnfext}, (i), (iii)]. Then our main result concerning a generalization of the Neukirch-Uchida theorem is as follows [cf.\ Corollary \ref{lqnfbij}]:
\begin{introthm}\label{introthmD}
Let $l$ be a prime number; $K_1,K_2$ $l$-quasi-number fields. Then the natural map
\begin{equation*}
\Isom(K_2,K_1)\to\OutIsom(G_{K_1},G_{K_2})
\end{equation*}
is bijective. In particular, $K_1$ and $K_2$ are isomorphic if and only if $G_{K_1}$ and $G_{K_2}$ are isomorphic [as profinite groups].
\end{introthm}
As a consequence of Theorem \ref{introthmD}, together with the well-known computation of cohomological dimensions, one may also obtain the following result [cf.\ Corollary \ref{cyclsubext}]:
\begin{introcor}\label{introcorE}
Write $\mathbb{Q}^{\mathrm{cyc}}$ for the cyclotomic $\widehat{\mathbb{Z}}$-extension of $\mathbb{Q}$; $\mathcal{F}$ for the set of finite extension fields of proper subfields of $\mathbb{Q}^{\mathrm{cyc}}$. Let $K_1,K_2\in\mathcal{F}$. Then the natural map
\begin{equation*}
\Isom(K_2,K_1)\to\OutIsom(G_{K_1},G_{K_2})
\end{equation*}
is bijective. In particular, $K_1$ and $K_2$ are isomorphic if and only if $G_{K_1}$ and $G_{K_2}$ are isomorphic [as profinite groups].
\end{introcor}
It is an important point of Corollary \ref{introcorE} that the $l$-dependence assumption on $K_1$, $K_2$ in Theorem \ref{introthmD} is eliminated in this special situation. At the time of writing of the present paper, the authors do not know whether or not the $l$-dependence assumption on $K_1$, $K_2$ in Theorem \ref{introthmD} can be eliminated in general. 

Finally, as mentioned briefly above, we explain a conceptual importance of the notion of a families preserving automorphism/isomorphism. For this, we make an observation concerning an analogy between our purely group-theoretic context and the usual situation in anabelian geometry whose ultimate goal is the reconstruction of geometric structures. Let $G$ be a profinite group; $\sigma$ an automorphism of $G$. Then it appears to the authors that the reconstruction of conjugacy classes of pro-cyclic subgroups of $G$ [i.e., the verification that $\sigma$ is families preserving] may be viewed as the group-theoretic counterpart of the reconstruction of underlying set of points of the geometric object. The further assertion that $\sigma$ is ultimately inner should then correspond to the reconstruction of ``rational" functions on the set of those points, which may be regarded as a group-theoretic version of the principle that ``point-theoretic implies geometric" that appears frequently in anabelian geometry. This analogy is particularly transparent in the setting of the Neukirch-Uchida theorem for number fields. In this setting, the first step is to recover the decomposition subgroups of finite primes, which are precisely the points of the spectrum of the ring of integers. On the other hand, we note that almost all decomposition subgroups are cyclic in finite quotients. Thus, one may observe the close relationship at the level of points. The subsequent passage from this point-theoretic information to the conclusion that a given isomorphism $\tau$ of the absolute Galois groups of number fields is induced by a field isomorphism precisely corresponds to the verification that $\tau$ arises from an inner automorphism of $G_{\mathbb{Q}}$. In a more abstract situation, one may consider, for each open subgroup $H \subseteq G$, the abstract space formed by the conjugacy classes of its distinguished pro-cyclic subgroups of $H$. In particular, one may observe that if $\sigma$ is families preserving, then $\sigma$ induces an automorphism of the resulting family of spaces equipped with the natural $G$-action, compatible with passage between open subgroups and inducing the identity at the bottom level. In this formulation, the reconstruction of the abstract point
spaces corresponds to the point-theoretic part, whereas the innerness conclusion plays the role
of reconstructing the rational functions on points, and hence the full geometric structure. The authors hope that this observation prompts further investigations of [highly nonabelian] profinite groups.\\

The present paper is organized as follows. In \S 1, we recall some definitions and results of the authors' previous work \cite{MST} that are applied in the present paper. In \S 2, we prove Theorem \ref{introthmA} and its strengthened version. In \S 3, as applications of the strengthened version of Theorem \ref{introthmA}, together with highly nontrivial group-theoretic/arithmetic results, we prove Theorems \ref{introthmB}, \ref{introthmC}. In this section, we also discuss group extensions of profinite groups that appear in Theorems \ref{introthmB}, \ref{introthmC}. In \S 4, we prove Theorem \ref{introthmD} and Corollary \ref{introcorE}.

\vskip.5\baselineskip
\section*{Notations and conventions}
\textbf{Numbers:} The notation $\mathbb{Z}$ will be used to denote the ring of integers. The notation $\mathbb{Q}$ will be used to denote the field of fractions of $\mathbb{Z}$. The notation $\widehat{\mathbb{Z}}$ will be used to denote the profinite completion of the underlying additive group of $\mathbb{Z}$. The notation $\mathbb{C}$ will be used to denote the field of complex numbers.

If $p$ is a prime number, then the notation $\mathbb{Z}_p$ will be used to denote the ring of $p$-adic integers; the notation $\mathbb{Q}_p$ will be used to denote the field of fractions of $\mathbb{Z}_p$; the notation $\mathbb{C}_p$ will be used to denote the $p$-adic completion of an algebraic closure of $\mathbb{Q}_p$; the notation $\mathbb{F}_p$ will be used to denote the finite field of cardinality $p$.

We shall refer to a finite extension field of $\mathbb{Q}$ as a \textit{number field}. We shall refer to a finite extension field of $\mathbb{Q}_p$ as a \textit{$p$-adic local field}.

\textbf{Fields:} Let $F$ be a field; $L\supseteq F$ an algebraic extension of $F$. Then we shall write $\overline{F}$ (respectively, $F^{\mathrm{sep}}$) for the algebraic closure (respectively, separable closure) of $F$ [determined up to isomorphisms]; $G_F\overset{\mathrm{def}}{=}\Gal(F^{\mathrm{sep}}/F)$; $[L:F]$ for the degree of $L$ over $F$ [here we regard this as a supernatural number].

Let $\Sigma$ be a set of prime numbers. Then we shall say that an algebraic extension $L/F$ is \textit{pro-prime-to-$\Sigma$} if $[M:F]$ is not divided by $l$ for each $l\in\Sigma$, where $M$ is the Galois closure of $F$.

Let $p$ be a prime number; $n$ a positive integer. Then we shall write $\mu_n$ for the set of $n$-th roots of unity in $\overline{\mathbb{Q}}$; $\mu_{p^\infty}\overset{\mathrm{def}}{=}\bigcup_m\mu_{p^m}$. We shall fix a primitive $n$-th root of unity $\zeta_n$.

If $F$ is an algebraic extension of $\mathbb{Q}$, then we shall write $\mathbb{V}(F)^{\mathrm{non}}$ for the set of nonarchimedean places of $F$. For $v\in\mathbb{V}(F)^{\mathrm{non}}$ and a subfield $K\subseteq F$ of $F$, we shall write $D_{F/K,v}$ for the decomposition subgroup of $\Gal(F/K)$ associated to $v$; $F_v$ for the completion of $F$ at $v$; $p_v$ for the residue characteristic of $F_v$. For $\overline{v}\in\mathbb{V}(\overline{\mathbb{Q}})^{\mathrm{non}}$, we shall write $D_{F,\overline{v}}\overset{\mathrm{def}}{=}D_{\overline{\mathbb{Q}}/F,\overline{v}}$.

Let $F_1, F_2$ be fields. Then we shall write $\Isom(F_1,F_2)$ for the set of all isomorphisms of fields from $F_1$ to $F_2$.

\textbf{Groups:} Let $G$ be a group; $H\subseteq G$ a subgroup; $n$ a nonnegative integer. Then we shall write $Z_G(H)$ for the \textit{centralizer} of $H$ in $G$, i.e., the subgroup $\{g\in G\,\vert\, ghg^{-1}=h\text{ for any }h\in H\}$; $Z(G)\overset{\mathrm{def}}{=}Z_G(G)$; $N_G(H)$ for the \textit{normalizer} of $H$ in $G$, i.e., the subgroup $\{g\in G\,\vert\, gHg^{-1}=H\}$; $G^{\wedge}$ for the profinite completion of $G$; $[G:H]$ for the index of $H$ in $G$. If $G$ is a profinite group and $H\subseteq G$ is a closed subgroup of $G$, then we regard the index $[G:H]$ as a supernatural number. We shall say that $G$ is \textit{center-free} if $Z(G)=\{1\}$. We shall say that $H$ is \textit{$n$-subnormal} in $G$ if there exist [not necessarily distinct] subgroups $H_0=G,H_1,\ldots,H_{n-1},H_n=H$ of $G$ such that $H_i$ is normal in $H_{i-1}$ for each $i\in\{1,\ldots,n\}$. [If $G$ is topological group and $H$ is an $n$-subnormal closed subgroup of $G$, then we may choose $H_0,H_1,\ldots,H_n$ to be closed [cf.\ \cite{MST}, Remark 1.1.1].] We shall say that $H$ is \textit{subnormal} in $G$ if $G$ is $m$-subnormal for some nonnegative integer $m$.

Suppose that $G$ is a profinite group. Then we shall say that $G$ is \textit{slim} if $Z_G(U)=\{1\}$ for every open subgroup $U$ of $G$, or, equivalently, every open subgroup of $G$ is center-free [cf.\ \cite{AbsTopI}, Notations and Conventions]. We shall refer to the minimum cardinality [possibly infinite] of topological generators of $G$ as \textit{rank} of $G$. We shall write $G^{\mathrm{ab}}$ for the quotient of $G$ by the closure of the commutator subgroup of $G$. If $p$ is a prime number, then we shall write $\cdim G$ for the cohomological dimension of $G$; $\cdim_p G$ for the $p$-cohomological dimension of $G$ [cf.\ \cite{NSW}, Definition 3.3.1].

Let $G_1,G_2$ be profinite groups. Then we shall write $\Isom(G_1,G_2)$ for the set of continuous isomorphisms of profinite groups from $G_1$ to $G_2$; $\OutIsom(G_1,G_2)$ for the set of continuous outer isomorphisms of profinite groups from $G_1$ to $G_2$, i.e., continuous isomorphisms considered up to composition with an inner automorphism arising from an element of $G_2$; $\Aut(G_1)\overset{\mathrm{def}}{=}\Isom(G_1,G_1)$; $\Out(G_1)\overset{\mathrm{def}}{=}\OutIsom(G_1,G_1)$.

Let $\mathcal{C}$ be a family of finite groups containing the trivial group. Then we shall refer to $\mathcal{C}$ as a \textit{full-formation} if $\mathcal{C}$ is closed under taking quotients, subgroups, and extensions.

\textbf{Fundamental groups:} Let $S$ be a connected locally Noetherian scheme. Then we shall write $\pi_1(S)$ for the \'etale fundamental group of $S$, relative to a suitable choice of basepoint. [Note that, for any field $F$, $\pi_1(\Spec F)\cong G_F$.] If $X$ is an algebraic variety [i.e., a separated, of finite type, and geometrically connected scheme] over $\mathbb{C}$, then we shall write $\pi_1^{\mathrm{top}}(X)$ for the topological fundamental group of the complex analytic space associated to $X$, relative to a suitable choice of [$\mathbb{C}$-rational] basepoint. If $K$ is a complete subfield of $\mathbb{C}_p$ and $X$ is a smooth variety over $K$, then we shall write $\pi_1^{\mathrm{temp}}(X)$ for the tempered fundamental group of $X$, relative to a suitable choice of basepoint [cf.\ \cite{An}].

\vskip.5\baselineskip
\section{Anabelian group-theoretic properties}
In the present section, we introduce terms and results from the authors' previous paper \cite{MST} that are also used in the present paper. While \cite{MST} sometimes deals with cases that are not necessarily profinite, here we refer only to profinite cases.

In the present section, let $p$ be a prime number.

\begin{defi}[cf.\ \cite{MST}, Definitions 1.5, 1.8; \cite{MST}, Proposition 1.7]\label{intindecomp}
Let $n$ be a positive integer; $G$ a profinite group.
\begin{enumerate}
\item Let $H\subseteq G$ be a closed subgroup. We shall say that $H$ is \textit{normally decomposable} in $G$ if there exist nontrivial normal closed subgroups $H_1,H_2\subseteq G$ of $G$ such that $H=H_1\times H_2$. We shall say that $H$ is \textit{normally indecomposable} in $G$ if $H$ is not normally decomposable in $G$. We shall say that $G$ is \textit{decomposable} (respectively, \textit{indecomposable}) if $G$ is normally decomposable (respectively, normally indecomposable) in $G$.
\item We shall say that $G$ is \textit{internally indecomposable} if every normal closed subgroup of $G$ is center-free and normally indecomposable in $G$, or, equivalently, $Z_G(H)=\{1\}$ for every nontrivial normal closed subgroup $H\subseteq G$. We shall say that $G$ is \textit{strongly internally indecomposable} if every open subgroup of $G$ is internally indecomposable.
\item We shall say that $G$ is \textit{$n$-sn-internally indecomposable} if every $(n-1)$-subnormal closed subgroup of $G$ is internally indecomposable, or, equivalently, $Z_G(H)=\{1\}$ for every nontrivial $n$-subnormal closed subgroup $H\subseteq G$. We shall say that $G$ is \textit{strongly $n$-sn-internally indecomposable} if every open subgroup of $G$ is $n$-sn-internally indecomposable.
\item We shall say that $G$ is \textit{sn-internally indecomposable} (respectively, \textit{strongly sn-internally indecomposable}) if $G$ is $m$-sn-internally indecomposable (respectively, strongly $m$-sn-internally indecomposable) for any positive integer $m$.
\end{enumerate}
\end{defi}

\begin{defi}[cf.\ \cite{MST}, Definition 2.1]\label{sne1}
Let $n$ be a positive integer; $G$ a profinite group.
\begin{enumerate}
\item We shall say that $G$ is \textit{$n$-sn-quasielastic} (respectively, \textit{$n$-sn-elastic}) if every topologically finitely generated $n$-subnormal closed subgroup of $G$ (respectively, of an open subgroup of $G$) is trivial or open in $G$.
\item We shall say that $G$ is \textit{sn-quasielastic} (respectively, \textit{sn-elastic}) if $G$ is $m$-sn-quasielastic (respectively, $m$-sn-elastic) for any positive integer $m$.
\item We shall say that $G$ is \textit{very $n$-sn-quasielastic} (respectively, \textit{very $n$-sn-elastic}; \textit{very sn-quasielastic}; \textit{very sn-elastic}) if $G$ is $n$-sn-quasielastic (respectively, $n$-sn-elastic; sn-quasielastic; sn-elastic), but not topologically finitely generated.
\end{enumerate}
\end{defi}

\begin{defi}[cf.\ \cite{FJ}, Definition 17.3.2; \cite{MoTa}, Definition 1.1, (iii)]\label{almostquot}
Let $\mathcal{C}$ be a full-formation; $G$ a profinite group; $Q$ a quotient of $G$ [in the category of profinite groups].
\begin{enumerate}
\item We shall write $G^{\mathcal{C}}$ for the maximal pro-$\mathcal{C}$ quotient of $G$. If $\mathcal{C}$ is the family of all $p$-groups, then we shall also write $G^p\overset{\mathrm{def}}{=}G^{\mathcal{C}}$.
\item Let $N\subseteq G$ be a normal open subgroup. If the kernel of the surjection $G\twoheadrightarrow Q$ coincides with the kernel of the natural surjection $N\twoheadrightarrow N^{\mathcal{C}}$ (respectively, $N\twoheadrightarrow N^p$), then we shall say that $Q$ is the \textit{almost pro-$\mathcal{C}$-maximal quotient} (respectively, \textit{almost pro-$p$-maximal quotient}) of $G$ associated to $N$.
\item We shall say that $Q$ is an \textit{almost pro-$\mathcal{C}$-maximal quotient} (respectively, \textit{almost pro-$p$-maximal quotient}) of $G$ if it is the almost pro-$\mathcal{C}$-maximal quotient (respectively, almost pro-$p$-maximal quotient) of $G$ associated to $N$ for some normal open subgroup $N\subseteq G$.
\end{enumerate}
\end{defi}

\begin{defi}[cf.\ \cite{H1}, Definition 2.1, (i)]\label{hypcurvedef}
Let $S$ be a scheme; $X$ a scheme over $S$. Then we shall say that $X$ is a \textit{hyperbolic curve} over $S$ if there exist nonnegative integers $g,r$, a scheme $X^{\mathrm{cpt}}$ over $S$, and a [possibly empty] closed subscheme $D\subseteq X^{\mathrm{cpt}}$ such that the following holds:
\begin{itemize}
\item $2g-2+r>0$;
\item $X^{\mathrm{cpt}}\to S$ is smooth, proper, geometrically connected, of relative dimension $1$, and any geometric fiber is [a necessarily smooth proper curve] of genus $g$;
\item the composite $D\hookrightarrow X^{\mathrm{cpt}}\to S$ is a finite \'etale morphism of degree $r$;
\item $X$ is isomorphic to $X^{\mathrm{cpt}}\setminus D$ over $S$.
\end{itemize}
We shall refer to $(g,r)$ as the \textit{type} of $X$. If $S$ is the spectrum of an algebraically closed field, then $D$ consists of $r$ points. In this case, we shall refer to each point of $D$ as a \textit{cusp}.
\end{defi}

\begin{defi}[cf.\ \cite{MoTa}, Definition 1.2]\label{surfacedef}
Let $\mathcal{C}$ be a full-formation; $\Pi$ a profinite group. Then we shall say that $\Pi$ is a \textit{pro-$\mathcal{C}$ surface group} (respectively, an \textit{almost pro-$\mathcal{C}$ surface group}) if $\Pi$ is isomorphic to the maximal pro-$\mathcal{C}$ quotient (respectively, an almost pro-$\mathcal{C}$-maximal quotient) of the \'etale fundamental group of a hyperbolic curve over an algebraically closed field of characteristic $0$. If $\mathcal{C}$ is the family of all $p$-groups, then we shall also refer to a pro-$\mathcal{C}$ surface group as a \textit{pro-$p$ surface group}.
\end{defi}

\begin{prop}[cf.\ \cite{MST}, Lemma 1.15]\label{snisom}
Let $n$ be a positive integer; $G$ an $n$-sn-internally indecomposable profinite group; $S\subseteq G$ a nontrivial $n$-subnormal closed subgroup; $H\subseteq G$ a closed subgroup containing $S$; $\alpha:H\to G$ a  continuous homomorphism. Suppose that for any $h\in S$, it holds that $\alpha(h)=h$. Then for any $g\in H$, it holds that $\alpha(g)=g$.
\end{prop}

\begin{prop}[cf.\ \cite{MST}, Proposition 1.11]\label{snintersect}
Let $n$ be a positive integer; $G$ a profinite group; $\{H_\lambda\}_{\lambda\in\Lambda}$ a set of $n$-subnormal closed subgroups of $G$. Then the following hold:
\begin{enumerate}
\item $\bigcap_{\lambda\in\Lambda}H_{\lambda}$ is $n$-subnormal in $G$.
\item Suppose that $G$ is $n$-sn-internally indecomposable, that $\Lambda$ is finite, and that for each $\lambda\in\Lambda$, $H_{\lambda}$ is nontrivial. Then $\bigcap_{\lambda\in\Lambda}H_{\lambda}$ is nontrivial.
\end{enumerate}
\end{prop}

\begin{prop}[cf.\ \cite{MST}, Proposition 1.14; \cite{MST}, Lemma 2.4]\label{projlim}
Let $n$ be a positive integer; $G$ a profinite group; $\{G_i\}_{i\in I}$ a directed subset of the set of normal closed subgroups of $G$ [where $j\ge i\Leftrightarrow G_j\subseteq G_i$] such that the natural homomorphism $G\to\varprojlim_{i\in I}G/G_i$ is an isomorphism. If for each $i\in I$, $G/G_i$ is $n$-sn-internally indecomposable (respectively, strongly $n$-sn-internally indecomposable; very $n$-sn-quasielastic; very $n$-sn-elastic), then so is $G$.
\end{prop}

\begin{prop}[cf.\ \cite{MST}, Proposition 1.12; \cite{MST}, Lemma 2.3]\label{snopen}
Let $n$ be a positive integer; $G$ a profinite group; $H\subseteq G$ an open subgroup. Then the following hold:
\begin{enumerate}
\item Suppose that any open subgroup of $G$ has no nontrivial finite normal subgroup [e.g.\ the case where $G$ is slim [cf.\ \cite{MiTs1}, Lemma 1.3]]. Then if $H$ is strongly $n$-sn-internally indecomposable, then so is $G$.
\item Suppose that $G$ has no nontrivial finite $n$-subnormal subgroup. Then if $H$ is $n$-sn-quasielastic (respectively, very $n$-sn-quasielastic), then so is $G$.
\end{enumerate}
\end{prop}

\begin{prop}[cf.\ \cite{MST}, Corollary 3.4]\label{pfreesne}
Every free pro-$p$ group [of arbitrary rank] is sn-elastic.
\end{prop}

\begin{prop}[cf.\ \cite{MST}, Proposition 3.6, (ii)]\label{Demushkin1}
Let $G$ be an infinite pro-$p$ Demu\v{s}kin group of rank $\ge 2$. Then every nontrivial closed subgroup of infinite index of $G$ is a free pro-$p$ group.
\end{prop}

\begin{prop}[cf.\ \cite{MST}, Lemma 3.9]\label{hilb1}
Let $K$ be a Hilbertian field [cf.\ \cite{FJ}, \S 12.1]; $H\subseteq G_K$ a nontrivial subnormal closed subgroup; $U\subseteq H$ a proper open subgroup of $H$. Then the separable extension of $K$ associated to $U\subseteq G_K$ is Hilbertian.
\end{prop}

\begin{thm}[cf.\ \cite{MST}, Theorems A, B]\label{sneanab}
Let $\mathcal{C}$ be a full-formation such that $\mathbb{Z}/p\mathbb{Z}$ is a $\mathcal{C}$-group; $G$ a profinite group. Suppose that $G$ is isomorphic to one of the following:
\begin{itemize}
\item an almost pro-$\mathcal{C}$-maximal quotient of a free profinite group of [possibly infinite] rank $\ge 2$;
\item an almost pro-$\mathcal{C}$ surface group;
\item a pro-$p$ Demu\v{s}kin group [cf.\ \cite{NSW}, Definition 3.9.9] of rank $\ge 3$;
\item an almost pro-$\mathcal{C}$-maximal quotient of $G_K$, where $K$ is a Henselian discrete valuation field [cf.\ \cite{FJ}, \S 11.5] of residue characteristic $p$ or a Hilbertian field.
\end{itemize} 
Then $G$ is strongly sn-internally indecomposable and sn-elastic.
\end{thm}

\vskip.5\baselineskip
\section{Certain isomorphisms between closed subgroups of profinite groups}
In the present section, in certain situations, we show that a certain isomorphism between closed subgroups of profinite groups is induced by an inner automorphism. Especially, families preserving isomorphisms [cf.\ Definition \ref{snfp}] play a crucial role [cf.\ Corollary \ref{innaut}].

In the present section, let $p$ be a prime number.

\begin{lem}\label{freestr}
Let $F$ be a free pro-$p$ group. Then the following hold:
\begin{enumerate}
\item Let $N\subseteq F$ be a normal open subgroup; $x\in F\setminus\ker(F\twoheadrightarrow F^{\mathrm{ab}}/pF^{\mathrm{ab}})$. Write $k$ for the minimum nonnegative integer such that $x^{p^k}\in N$; $a\overset{\mathrm{def}}{=}x^{p^k}$. Then there exist a $p$-power $m$ and $F$-conjugates $a_1=a,\ldots,a_m\in N$ of $a$ such that, if we write $e_i$ for the image of $a_i$ in $N^{\mathrm{ab}}$, then
\begin{itemize}
\item $e_1,\ldots,e_m\in N^{\mathrm{ab}}$ are linearly independent over $\mathbb{Z}_p$;
\item the image of any $F$-conjugate of $a$ in $N^{\mathrm{ab}}$ coincides with some $e_i$.
\end{itemize}
\item Let $x\in F\setminus\ker(F\twoheadrightarrow F^{\mathrm{ab}}/pF^{\mathrm{ab}})$. Write $I\subseteq F$ for the [pro-cyclic] subgroup of $F$ topologically generated by $x$. Let $y\in F$ be an element; $\{N_\lambda\}_{\lambda\in\Lambda}$ a set of normal open subgroups of $F$ such that $\bigcap_{\lambda\in\Lambda}N_{\lambda}=\{1\}$. Suppose that for any $\lambda\in\Lambda$, it holds that $\Image(I\cap N_{\lambda}\to N_{\lambda}^{\mathrm{ab}}/pN_{\lambda}^{\mathrm{ab}})=\Image(yIy^{-1}\cap N_{\lambda}\to N_{\lambda}^{\mathrm{ab}}/pN_{\lambda}^{\mathrm{ab}})$. Then it holds that $I=yIy^{-1}$.
\item Let $I\subseteq F$ be as in (ii). Then it holds that $N_F(I)=I$.
\end{enumerate}
\end{lem}

\begin{proof}
Since any free pro-$p$ group is a direct limit of free pro-$p$ groups of finite rank, we may assume that $F$ is of finite rank, which we denote by $r$. The case where $r\le 1$ is clear. Thus, we may further assume that $r\ge 2$. Let us fix a hyperbolic curve $X$ of type $(0,r+1)$ over $\mathbb{C}$ and an isomorphism $F\overset{\sim}{\to}\pi_1(X)^p$ such that $x\in F$ maps to a topological generator of a cuspidal inertia subgroup of the pro-$p$ surface group $\pi_1(X)^p$ [i.e., a closed subgroup corresponding to a cusp of $X$]. Then $N$ in assertion (i) (respectively, for each $\lambda\in\Lambda$, $N_\lambda$ in assertion (ii)) corresponds to a covering $Y$ of $X$ (respectively, $Y_\lambda$). Then, by writing $a_1=a,\ldots,a_m$ for the [finite] orbit of $a$ under the conjugation action of $F$, assertions (i), (ii) are well-known as properties of cuspidal inertia subgroups of pro-$p$ surface groups. Assertion (iii) is also well-known [cf.\ \cite{CmbGC}, Proposition 1.2, (ii)], or alternatively, since $N_F(I)\subseteq F$ is a free pro-$p$ group with pro-cyclic normal closed subgroup $I\subseteq N_F(I)$, it follows from Proposition \ref{pfreesne} that $N_F(I)$ is pro-cyclic, hence $N_F(I)=I$ by our choice of $I$. This completes the proof of Lemma \ref{freestr}.
\end{proof}

\begin{lem}\label{inertia}
Let $F$ be a free pro-$p$ group; $x,y\in F\setminus\ker(F\twoheadrightarrow F^{\mathrm{ab}}/pF^{\mathrm{ab}})$; $\{N_\lambda\}_{\lambda\in\Lambda}$ a set of normal open subgroups of $F$ such that $\bigcap_{\lambda\in\Lambda}N_{\lambda}=\{1\}$. Write $I_x,I_y$ for the pro-cyclic subgroups topologically generated by $x,y$, respectively. Suppose that for each $\lambda\in\Lambda$, the respective images of $I_x\cap N_{\lambda}$ and $I_y\cap N_{\lambda}$ via the natural surjection $N_{\lambda}\twoheadrightarrow N_{\lambda}^{\mathrm{ab}}/pN_{\lambda}^{\mathrm{ab}}$ coincide. Then it holds that $I_x=I_y$.
\end{lem}

\begin{proof}
For each $\lambda\in\Lambda$, since $yI_xy^{-1}\cap N_{\lambda}=y(I_x\cap N_{\lambda})y^{-1}$ and $y(I_y\cap N_{\lambda})y^{-1}=I_y\cap N_{\lambda}$, it follows from our assumption that the respective images of $y^{-1}I_xy\cap N_{\lambda}$ and $I_x\cap N_{\lambda}$ via the natural surjection $N_{\lambda}\twoheadrightarrow N_{\lambda}^{\mathrm{ab}}/pN_{\lambda}^{\mathrm{ab}}$ coincide. Thus, it follows from Lemma \ref{freestr}, (ii), that $yI_xy^{-1}=I_x$. Then it holds that $y\in N_F(I_x)=I_x$ [cf.\ Lemma \ref{freestr}, (iii)], which implies that $I_y\subseteq I_x$. Similarly, it holds that $I_x\subseteq I_y$. This completes the proof of Lemma \ref{inertia}.
\end{proof}

\begin{thm}\label{autfree}
Let $G$ be a profinite group; $F\subseteq G$ a normal closed subgroup of $G$; $\alpha\in\Aut(F)$ a continuous automorphism. Suppose that
\begin{itemize}
\item $F$ is a free pro-$p$ group of [possibly infinite] rank $\ge 2$;
\item $\alpha(N)=N$ for every closed subgroup $N\subseteq F$ of $F$ that is normal in $G$;
\item for every open subgroup $N\subseteq F$ that is normal in $G$, if $I$ is a pro-cyclic subgroup of $N$ such that $I\not\subseteq\ker(N\twoheadrightarrow N^{\mathrm{ab}}/pN^{\mathrm{ab}})$, then there exists $g\in G$ such that $\alpha(I)=gIg^{-1}$.
\end{itemize}
Then $\alpha$ is induced by an inner automorphism of $G$.
\end{thm}

\begin{proof}
Let us fix an element $x\in F\setminus\ker(F\twoheadrightarrow F^{\mathrm{ab}}/pF^{\mathrm{ab}})$. Write $I\subseteq F$ for the [pro-cyclic] subgroup topologically generated by $x$. For each $g\in G$ and each open subgroup $N\subseteq F$ of $F$ that is normal in $G$, we shall write $\alpha_g\in\Aut(F)$ for the automorphism of $F$ determined by $\alpha_g(h)=g^{-1}\alpha(h)g$; $\alpha_{g,N}^{\mathrm{ab}}$ for the automorphism of $N^{\mathrm{ab}}$ determined by $\alpha_g$; $M_N\subseteq N^{\mathrm{ab}}$ for the subgroup of $N^{\mathrm{ab}}$ generated by the image of all $F$-conjugates of $I\cap N$. First, we claim the following:
\begin{quotation}\hypertarget{autfreeclaim}{}
Claim \ref*{autfree}.A: Let $N_0\subseteq F$ be an open subgroup of $F$ that is normal in $G$. Then there exist $g\in G$ and $b\in\mathbb{Z}_p^{\times}$ such that for any $a\in M_{N_0}$, it holds that $\alpha_{g,N_0}^{\mathrm{ab}}(a)=ba$.
\end{quotation}

Indeed, we shall write $k$ for the unique nonnegative integer such that $I\cap N_0$ is topologically generated by $x^{p^k}$. Then it follows from Lemma \ref{freestr}, (i), that there exist $F$-conjugates $x_1=x^{p^k},x_2,\ldots,x_m$ of $x^{p^k}$ such that, if we write $e_i$ for the image of $x_i$ in $N_0^{\mathrm{ab}}$, then $\{e_1,\ldots,e_m\}$ is a linearly independent generator of $M_{N_0}$ over $\mathbb{Z}_p$, and, moreover, it holds that the image of any $F$-conjugate of $x^{p^k}$ in $N_0^{\mathrm{ab}}$ is in $\{e_1,\ldots,e_m\}\subseteq M_{N_0}$. Write $I_i\subseteq N_0$ for the [pro-cyclic] subgroup topologically generated by $x_i$.

Now for each open subgroup $N\subseteq N_0$ of $N_0$ that is normal in $G$, we shall write $k_N$ for the unique nonnegative integer such that $I\cap N$ is topologically generated by $x^{p^{k_N}}$. Then $I_i\cap N$ is topologically generated by $x_i^{p^{k_N-k}}$. For each $i\in\{1,\ldots,m\}$, we shall write $e_{i,N}\in N^{\mathrm{ab}}$ for the image of $x_i^{p^{k_N-k}}$ in $N^{\mathrm{ab}}$. Moreover, we shall write
\begin{equation*}
z_N\overset{\mathrm{def}}{=}\sum_{i=1}^m p^{i-1}e_{i,N}\in N^{\mathrm{ab}};
\end{equation*}
\begin{equation*}
B_N\overset{\mathrm{def}}{=}\{(b,g)\in\mathbb{Z}_p^{\times}\times G\,\vert\,\alpha_{g,N}^{\mathrm{ab}}(z_N)=bz_N\}\subseteq\mathbb{Z}_p^{\times}\times G.
\end{equation*}
Then it follows from our assumption that $B_N$ is a nonempty [closed] subset of $\mathbb{Z}_p^{\times}\times G$. Moreover, if $N'\subseteq N\subseteq N_0$ are open subgroups of $N_0$ that is normal in $G$, then, since the image of $z_{N'}$ via the natural homomorphism $N'^{\mathrm{ab}}\to N^{\mathrm{ab}}$ is $p^{k_{N'}-k_N}z_N$, it holds that $B_{N'}\subseteq B_N$. Thus, since $\mathbb{Z}_p^{\times}\times G$ is compact, it holds that $\bigcap_N B_N\neq\emptyset$.

Now let $(b,g)\in\bigcap_N B_N$. Note that the image of $z_N$ via the natural surjection $N^{\mathrm{ab}}\twoheadrightarrow N^{\mathrm{ab}}/pN^{\mathrm{ab}}$ coincides with the image of $x^{p^{k_N}}$ via the natural surjection $N\twoheadrightarrow N^{\mathrm{ab}}/pN^{\mathrm{ab}}$. Thus, it follows from Lemma \ref{inertia} that $\alpha_g(I_1)=I_1$. In particular, there exists $c\in\mathbb{Z}_p^{\times}$ such that $\alpha_g(e_1)=ce_1$, and, moreover, for each $i\in\{1,\ldots,m\}$, there exists $j_i\in\{1,\ldots,m\}$ such that $\alpha_{g,N_0}(e_i)=ce_{j_i}$. Then it holds that
\begin{equation*}
\sum_{i=1}^m bp^{i-1}e_i=bz_{N_0}=\alpha_{g,N_0}^{\mathrm{ab}}(z_{N_0})=\sum_{i=1}^m cp^{i-1}e_{j_i},
\end{equation*}
which implies that $b=c$ and $j_i=i$. Since $M_{N_0}$ is generated by $e_1,\ldots,e_m$, we conclude that $\alpha_{g,N_0}^{\mathrm{ab}}(a)=ba$ for any $a\in M_{N_0}$. This completes the proof of Claim \hyperlink{autfreeclaim}{\ref*{autfree}.A}.

Now for each open subgroup $N\subseteq F$ of $F$ that is normal in $G$, we shall write
\begin{equation*}
C_N\overset{\mathrm{def}}{=}\{(b,g)\in\mathbb{Z}_p^{\times}\times G\,\vert\,\alpha_{g,N}^{\mathrm{ab}}(a)=ba\text{ for all }a\in M_N\}\subseteq\mathbb{Z}_p^{\times}\times G.
\end{equation*}
Then it follows from Claim \hyperlink{autfreeclaim}{\ref*{autfree}.A} that $C_N$ is a nonempty [closed] subset of $\mathbb{Z}_p^{\times}\times G$. Moreover, if $N'\subseteq N\subseteq F$ are open subgroups of $F$ that are normal in $G$, then it holds that $C_{N'}\subseteq C_N$. Indeed, let $(b,g)\in C_{N'}$. For any $a\in M_N$, there exists $r\in\mathbb{Z}_p\setminus\{0\}$ such that $ra$ is in the image of $M_{N'}\subseteq N'^{\mathrm{ab}}$ via the natural homomorphism $N'^{\mathrm{ab}}\to N^{\mathrm{ab}}$. Then, since it holds that $\alpha_{g,N}^{\mathrm{ab}}(ra)=bra$, it follows from the torsion-freeness of $\mathbb{Z}_p$-module $N^{\mathrm{ab}}$ that $\alpha_{g,N}^{\mathrm{ab}}(a)=ba$, i.e., $(b,g)\in C_N$.

Since $\mathbb{Z}_p^{\times}\times G$ is compact, it holds that $\bigcap_N C_N\neq\emptyset$. Let $(b,g)\in\bigcap_N C_N$. Write $A$ for the set of $F$-conjugates of $I$. Then it follows from Lemma \ref{freestr}, (ii), that $\alpha_g(J)=J$ for any $J\in A$.

Now for any $h\in F$ and $J\in A$, since $h^{-1}Jh\in A$, it holds that
\begin{equation*}
\alpha_g(h)h^{-1}J(\alpha_g(h)h^{-1})^{-1}=\alpha_g(h)\alpha_g(h^{-1}Jh)\alpha_g(h)^{-1}=\alpha_g(J)=J.
\end{equation*}
Thus, it holds that $\alpha_g(h)h^{-1}\in\bigcap_{J\in A}N_F(J)=\bigcap_{J\in A}J=\{1\}$ [cf.\ Lemma \ref{freestr}, (iii)], which implies that $\alpha_g=\id_F$. This completes the proof of Theorem \ref{autfree}.
\end{proof}

\begin{thm}\label{innisom}
Let $n$ be a positive integer, $G$ an $n$-sn-internally indecomposable profinite group; $H,H',F,\widetilde{F}\subseteq G$ closed subgroups of $G$ such that $F\subseteq H\cap H'\cap\widetilde{F}$; $\alpha:H\overset{\sim}{\to}H'$ a continuous isomorphism. Suppose that
\begin{itemize}
\item $\widetilde{F}$ is $(n-1)$-subnormal in $G$;
\item $F$ is a nontrivial free pro-$p$ normal subgroup of $\widetilde{F}$;
\item $\alpha(N)=N$ for every closed subgroup $N\subseteq F$ of $F$ that is normal in $\widetilde{F}$;
\item for every open subgroup $N\subseteq F$ that is normal in $\widetilde{F}$, if $I$ is a pro-cyclic subgroup of $N$ such that $I\not\subseteq\ker(N\twoheadrightarrow N^{\mathrm{ab}}/pN^{\mathrm{ab}})$, then there exists $g\in\widetilde{F}$ such that $\alpha(I)=gIg^{-1}$.
\end{itemize}
Then $\alpha$ is induced by an inner automorphism of $\widetilde{F}$.
\end{thm}

\begin{proof}
If there exists $g\in\widetilde{F}$ such that $g^{-1}\alpha(h)g=h$ for any $h\in F$, then it follows from Proposition \ref{snisom} that $g^{-1}\alpha(h)g=h$ for any $h\in H$. Thus, by replacing $(n,G,H,H',\alpha)$ by $(1,\widetilde{F},F,F,\alpha\vert_F:F\overset{\sim}{\to}F)$, we may assume that $n=1$, $G=\widetilde{F}$, $H=H'=F$. Since $G$ is internally indecomposable, the [nontrivial] free pro-$p$ normal closed subgroup $F$ is center-free, hence of rank $\ge 2$. Thus, Theorem \ref{innisom} follows from Theorem \ref{autfree}.
\end{proof}

\begin{rem}\label{innisomrem}
\mbox{}
\begin{enumerate}
\item If $G=F$ (hence $G$ itself is a free pro-$p$ group of rank $\ge 2$), then the two theorems above say that any normal automorphism $\alpha$ of $G$ satisfying the third assumption of Theorem \ref{autfree} is an inner automorphism. On the other hand, by a similar argument to the argument of \cite{NodNon}, Lemma 1.6, we can show that any normal automorphism $\alpha$ of $G$ satisfies the third assumption of Theorem \ref{autfree}. This gives an alternative proof of \cite{JR}, Corollary C in the case where ``$\mathcal{C}$'' in \cite{JR}, Corollary C is the class of all $p$-groups.

Notice that the proof in \cite{JR} is based on representation theoretic considerations, while the proof of Theorem \ref{autfree} is based on anabelian geometric considerations.
\item Theorem \ref{innisom} is a result that provides evidence for the importance of considering [sn-]internal indecomposability.
\item It is natural to pose the following question:
\begin{quotation}
\noindent Question: Can the third assumption of Theorem \ref{autfree} be dropped?
\end{quotation}
However, at the time of writing of the present paper, the authors do not know whether this question is affirmative or not.
\end{enumerate}
\end{rem}

\begin{lem}\label{innlim}
Let $G$ be a profinite group; $H,H'\subseteq G$ closed subgroups; $\alpha:H\to H'$ a continuous homomorphism; $\{G_i\}_{i\in I}$ a directed subset of the set of normal closed subgroups of $G$ [where $j\ge i\Leftrightarrow G_j\subseteq G_i$] such that the natural homomorphism $G\to\varprojlim_{i\in I}G/G_i$ is an isomorphism. For each $i\in I$, write $\phi_i:G\twoheadrightarrow G/G_i$ for the natural surjection. Suppose that $\alpha(H\cap G_i)\subseteq G_i$ for each $i\in I$, and write $\alpha_i:\phi_i(H)\to\phi_i(H')$ for the homomorphism induced by $\alpha$. Then the following conditions are equivalent:
\renewcommand{\labelenumi}{(\arabic{enumi})}
\begin{enumerate}
\item $\alpha$ is induced by an inner automorphism of $G$.
\item For any $i\in I$, $\alpha_i$ is induced by an inner automorphism of $G/G_i$.
\end{enumerate}
\renewcommand{\labelenumi}{(\roman{enumi})}
\end{lem}

\begin{proof}
The implication $(1)\Rightarrow(2)$ is immediate. We verify the implication $(2)\Rightarrow(1)$. Suppose that condition $(2)$ is satisfied. For each $i\in I$, we shall write
\begin{equation*}
D_i\overset{\mathrm{def}}{=}\{g\in G\,\vert\, \alpha_i(h)=\phi_i(g)h\phi_i(g)^{-1}\text{ for all }h\in\phi_i(H)\}\subseteq G.
\end{equation*}
Then $D_i$ is a nonempty [closed] subset of $G$. Moreover, if $j\ge i$, then it holds that $D_j\subseteq D_i$. Thus, since $G$ is compact, it holds that $\bigcap_{i\in I}D_i\neq\emptyset$.

Let $g\in\bigcap_{i\in I}D_i$. Then for any $h\in H$ and $i\in I$, it holds that $\phi_i(\alpha(h))=\alpha_i(\phi_i(h))=\phi_i(g)\phi_i(h)\phi_i(g)^{-1}=\phi_i(ghg^{-1})$, which implies that $\alpha(h)^{-1}ghg^{-1}\in\bigcap_{i\in I} G_i=\{1\}$. This completes the proof of Lemma \ref{innlim}.
\end{proof}

\begin{thm}\label{innisom2}
Let $G,H,H',\alpha,\{G_i\}_{i\in I},\phi_i,\alpha_i$ be as in Lemma \ref{innlim}; $n$ a positive integer. Suppose that $\alpha$ is an isomorphism. Moreover, for each $i\in I$, suppose that
\begin{itemize}
\item $\alpha(H\cap G_i)=H'\cap G_i$ [hence $\alpha_i$ is an isomorphism];
\item $G/G_i$ is $n$-sn-internally indecomposable;
\item $\phi_i(H)\cap \phi_i(H')$ contains a nontrivial free pro-$p_i$ $n$-subnormal subgroup $F_i$ of $G/G_i$ for some prime number $p_i$;
\item $\alpha_i(S)=S$ for every closed subgroup $S\subseteq F_i$ of $F_i$ that is $n$-subnormal in $G/G_i$;
\item for every open subgroup $U\subseteq F_i$ that is $n$-subnormal in $G/G_i$, if $I$ is a pro-cyclic subgroup of $U$ such that $I\not\subseteq\ker(U\twoheadrightarrow U^{\mathrm{ab}}/pU^{\mathrm{ab}})$, then there exists $g\in G/G_i$ such that $\alpha_i(I)=gIg^{-1}$.
\end{itemize}
Then $\alpha$ is induced by an inner automorphism of $G$.
\end{thm}

\begin{proof}
This follows from Theorem \ref{innisom}; Lemma \ref{innlim}.
\end{proof}

\begin{defi}\label{snfp}
Let $n$ be a positive integer; $G$ a profinite group; $H,H'\subseteq G$ closed subgroups; $\alpha:H\overset{\sim}{\to}H'$ a continuous isomorphism.
\begin{enumerate}
\item We shall say that $\alpha$ is \textit{$n$-subnormal} [in $G$] if for every closed subgroup $S\subseteq H\cap H'$ that is $n$-subnormal in $G$, it holds that $\alpha(S)=S$. We shall say that $\alpha$ is \textit{normal} [in $G$] if $\alpha$ is $1$-subnormal [in $G$].
\item (cf.\ \cite{JR}, Introduction) We shall say that $\alpha$ is \textit{families preserving} [in $G$] if for every pro-cyclic subgroup $I\subseteq H$, there exists $g\in G$ such that $\alpha(I)=gIg^{-1}$.
\end{enumerate}
\end{defi}

\begin{rem}\label{snfprem}
\mbox{}
\begin{enumerate}
\item If $H$ and $H'$ are open in $G$, then $\alpha$ is $n$-subnormal in $G$ if and only if $\alpha(S)=S$ for every open subgroup $S\subseteq H\cap H'$ that is $n$-subnormal in $G$.
\item Every families preserving isomorphism in $G$ is normal in $G$.
\item Suppose that $\alpha$ is a families preserving (respectively, an $n$-subnormal) isomorphism in $G$. Then so is $\alpha^{-1}$. Moreover, for any normal closed subgroup $N\subseteq G$ (respectively, normal closed subgroup $N\subseteq G$ contained in $H\cap H'$), the isomorphism $H/(H\cap N)\overset{\sim}{\to}H'/(H'\cap N)$ determined by $\alpha$ [cf.\ (ii)] is families preserving (respectively, $n$-subnormal) in $G/N$.
\end{enumerate}
\end{rem}

\begin{cor}\label{innaut}
Let $n$ be a positive integer; $G$ a profinite group; $H,H'\subseteq G$ closed subgroups; $\alpha:H\overset{\sim}{\to}H'$ a continuous isomorphism; $\{G_i\}_{i\in I}$ a directed subset of the set of closed subgroups of $H\cap H'$ that is normal in $G$ [where $j\ge i\Leftrightarrow G_j\subseteq G_i$] such that the natural homomorphism $G\to\varprojlim_{i\in I}G/G_i$ is an isomorphism. Suppose that for each $i\in I$,
\begin{itemize}
\item $\alpha$ is families preserving and $n$-subnormal in $G$.
\item $G/G_i$ is $n$-sn-internally indecomposable;
\item $(H\cap H')/G_i$ has a nontrivial free pro-$p_i$ closed subgroup that is $n$-subnormal in $G/G_i$ for some prime number $p_i$;
\end{itemize}
Then $\alpha$ is induced by an inner automorphism of $G$.
\end{cor}

\begin{proof}
This follows from Theorem \ref{innisom2}; Remark \ref{snfprem}, (iii).
\end{proof}

\vskip.5\baselineskip
\section{Families preserving isomorphisms for groups in anabelian geometry}\label{anabfp}
When applying Theorem \ref{innisom2} or Corollary \ref{innaut}, for a given profinite group, it is important to consider its quotients that are [sn-]internally indecomposable and have free pro-$p$ [sub]normal closed subgroups. In the present section, we discuss this problem for various groups appearing in anabelian geometry, and show that in many cases every families preserving isomorphism is induced by an inner automorphism.

In the present section, let $p$ be a prime number.

First, we deal with the absolute Galois groups and their quotients of various fields. We also provide an application to a family of fields that a weak version of Neukirch-Uchida type result [i.e., the isomorphism class of the absolute Galois group determines the field] were obtained for [cf.\ Remark \ref{uchida}].

\begin{lem}\label{normalize}
Let $G$ be a profinite group; $H\subseteq G$ a closed (respectively, a normal closed) subgroup; $N\subseteq H$ a normal open subgroup of $H$. Then there exist an open (respectively, a normal open) subgroup $U\subseteq G$ of $G$ containing $H$ and a normal open subgroup $V\subseteq U$ of $U$ such that $N=H\cap V$ and that $H/N\overset{\sim}{\to}U/V$.
\end{lem}

\begin{proof}
It follows from \cite{FJ}, Lemma 1.2.5, (b), that there exists an open subgroup $W\subseteq G$ of $G$ such that $N=H\cap W$. Write $V=\bigcap_{h\in H}hWh^{-1}$. Since the number of $G$-conjugates of $W$, hence also the number of $H$-conjugates of $W$, is finite, $V$ is open in $G$. Moreover, for $h\in H$, since $N=hNh^{-1}\subseteq hWh^{-1}$, it holds that $N\subseteq V$. In particular, it holds that $N=H\cap V$.

Now it follows from our choice of $V$ that $H\subseteq N_G(V)$. Thus, we obtain a natural injective homomorphism $H/N\hookrightarrow N_G(V)/V$. Moreover, if $H$ is normal in $G$, then, since $H$ is normal in $N_G(V)$, the image of the natural injective homomorphism $H/N\hookrightarrow N_G(V)/V$ is normal in $N_G(V)/V$. Thus, if we write $U\subseteq N_G(V)$ for the inverse image of $\Image(H/N\hookrightarrow N_G(V)/V)$ via the natural surjection $N_G(V)\twoheadrightarrow N_G(V)/V$, then the pair $(U,V)$ satisfies the desired conditions.
\end{proof}

\begin{prop}\label{nffree}
Let $K$ be a number field; $N\subseteq G_K$ a normal closed subgroup of $G_K$ such that $[G_K:N]$ is not divided by $p^\infty$. Write $Q\overset{\mathrm{def}}{=}G_K/\ker(N\twoheadrightarrow N^p)$. Then the following holds:
\begin{enumerate}
\item $Q$ is strongly sn-internally indecomposable, very sn-elastic, and $Q$ has a nontrivial free pro-$p$ normal closed subgroup. In particular, any almost pro-$p$-maximal quotient of $G_K$ has a nontrivial free pro-$p$ normal closed subgroup.
\item Let $H,H'\subseteq Q$ be closed subgroups. Suppose that $H,H'$ have nontrivial closed subgroups that are normal in $G$. Then every families preserving isomorphism $H\overset{\sim}{\to}H'$ in $G$ is induced by an inner automorphism of $G$.
\end{enumerate}
\end{prop}

\begin{proof}
First, we verify assertion (i). It follows from Lemma \ref{normalize} that $\ker(N\twoheadrightarrow N^p)=\bigcap_U\ker(U\twoheadrightarrow U^p)$, where $U$ runs over all normal open subgroups of $G_K$ containing $N$. Thus, since $U^p$, hence also $G_K/\ker(U\twoheadrightarrow U^p)$, is not topologically finitely generated, it follows from Proposition \ref{projlim}; Theorem \ref{sneanab}, that $Q$ is strongly sn-internally indecomposable and very sn-elastic. Thus, it suffices to show that $Q$ has a nontrivial free pro-$p$ normal closed subgroup.

Let $U\subseteq Q$ be a normal open subgroup of $Q$. If $U$ has a nontrivial free pro-$p$ normal closed subgroup $F\subseteq U$, then it follows from Proposition \ref{snintersect}; Theorem \ref{sneanab}, that $\bigcap_{g\in Q}gFg^{-1}\subseteq Q$ is a nontrivial [free pro-$p$] normal closed subgroup of $Q$. Thus, by replacing $K$ by the finite [Galois] extension of $K$ corresponding to a suitable normal open subgroup of $Q$, to verify Proposition \ref{nffree}, we may assume that $K$ is totally imaginary if $p=2$.

Write $L$ for the cyclotomic $\mathbb{Z}_p$-extension of $K$; $M$ for the [Galois] extension of $K$ obtained by compositing $L$ and the [Galois] extension of $K$ corresponding to $N$. Then it follows from \cite{NSW}, Corollary 8.1.18, that $\cdim_p G_M\le 1$. Thus, it follows from \cite{Se}, Chapter II, \S 2, Proposition 2, that $\cdim_p G_M^p\le 1$. On the other hand, since $[G_K:N]$ is not divided by $p^\infty$, it is clear that $G_M$ has a normal open subgroup of index a power of $p$, which implies that $\cdim_p G_M^p\ge 1$. Thus, we conclude that $G_M^p$ is a nontrivial free pro-$p$ group. Now since $[N:G_M]$ is not divided by all prime numbers except $p$, the image of $G_M$ via $N\twoheadrightarrow N^p$, hence also via $G_K\twoheadrightarrow Q$, is identified with $G_M^p$. Thus, $Q$ has a nontrivial free pro-$p$ normal closed subgroup. This completes the proof of assertion (i).

Next, we verify assertion (ii). It follows from assertion (i), together with Proposition \ref{snintersect}, that there exists a nontrivial free pro-$p$ normal closed subgroup $F\subseteq Q$ of $Q$ contained in $H\cap H'$. Thus, it follows from Theorem \ref{innisom}; Remark \ref{snfprem}, (ii), that every families preserving isomorphism $H\overset{\sim}{\to}H'$ in $Q$ is induced by an inner automorphism of $Q$. This completes the proof of assertion (ii), hence also of Proposition \ref{nffree}.
\end{proof}

\begin{prop}\label{henselfree}
Let $K$ be a Henselian discrete valuation field of residue characteristic $p$; $N\subseteq G_K$ a normal open subgroup of $G_K$. Then the almost pro-$p$-maximal quotient of $G_K$ associated to $N$ has a nontrivial free pro-$p$ normal closed subgroup.
\end{prop}

\begin{proof}
If $K$ is of characteristic $p$, then $N^p$ is a free pro-$p$ group of infinite rank. Thus, we may assume that $K$ is of characteristic $0$. Write $Q$ for the almost pro-$p$-maximal quotient of $G_K$ associated to $N$. If $N^p$ has a nontrivial free pro-$p$ normal closed subgroup $F$, then, since $N^p$ is a normal open subgroup of $Q$, it follows from Proposition \ref{snintersect}; Theorem \ref{sneanab}, that $\bigcap_{g\in Q}gFg^{-1}$ is a nontrivial free pro-$p$ normal closed subgroup of $Q$. Thus, by replacing $K$ by the finite Galois extension of $K$ corresponding to $N$, it suffices to show that $G_K^p$ has a nontrivial free pro-$p$ normal closed subgroup.

Now there exists a finite [Galois] subextension $L$ of the cyclotomic $\mathbb{Z}_p$-extension of $K$ such that the absolute ramification index of $L$ is $\ge p+1$. If $G_L^p$ has a nontrivial free pro-$p$ normal closed subgroup $F$, then it follows from Proposition \ref{snintersect}; Theorem \ref{sneanab}, that $\bigcap_{g\in G_K^p}gFg^{-1}$ is a nontrivial free pro-$p$ normal closed subgroup of $G_K^p$. Thus, by replacing $K$ by $L$, we may assume that the absolute ramification index of $K$ is $\ge p+1$.

Now we claim the following:

\begin{quotation}\hypertarget{henselfreeclaim}{}
Claim \ref*{henselfree}.A: There exists a weakly unramified pro-$p$ extension $L$ of $K$ such that the residue field of $L$ is perfect, and that $G_L^p$ has a nontrivial closed subgroup that is normal in $G_K^p$.
\end{quotation}

Indeed, let us fix a nontrivial element $g\in G_K^p$. Write $k$ for the residue field of $K$. Let $\{t_i\in k\,\vert\, i\in I\}$ be a $p$-basis of $k$. We may assume that $I$ is nonempty. For each $(i,j)\in I\times\mathbb{Z}_{\ge 0}$, write $k_{i,j}$ for the field extension of $k$ generated by $p^j$-th roots of $t_i$ over $k$ [hence $k_{i,0}=k$].

Let $i\in I$. We prove that there exists a sequence of finite weakly unramified $p$-extensions $K_{i,0}\subseteq K_{i,1}\subseteq K_{i,2}\subseteq\cdots$ of $K$ satisfying the following for each $j\ge 0$:
\begin{itemize}
\item the residue field of $K_{i,j}$ is $k_{i,j}$;
\item $g\in\bigcap_{\sigma\in G_K^p}\sigma G_{K_{i,j}}^p\sigma^{-1}$.
\end{itemize}
Write $K_{i,0}\overset{\mathrm{def}}{=}K$.  For each $j\ge 1$, suppose that we have already obtained $K_{i,j-1}$ satisfying the desired conditions. Let $\{\sigma_1,\ldots,\sigma_m\}$ be a left transversal of $G_{K_{i,j-1}}^p$ in $G_K^p$. Write $F\subseteq G_K^p$ for the closed subgroup of $G_K^p$ topologically generated by $\sigma_1^{-1}g\sigma_1,\ldots,\sigma_m^{-1}g\sigma_m$. Then, since we assume that $g\in\bigcap_{\sigma\in G_K^p}\sigma G_{K_{i,j-1}}^p\sigma^{-1}$, it holds that $F\subseteq G_{K_{i,j-1}}^p$. Thus, in light of our assumption on the absolute ramification index of $K$, by applying \cite{MiTs2}, Lemma 2.4, where we take ``$K$'' to be $K_{i,j-1}$, we obtain a weakly unramified Galois extension $K_{i,j}$ of $K_{i,j-1}$ of degree $p$ such that the residue field of $K_{i,j}$ is $k_{i,j}$, and that $F\subseteq G_{K_{i,j}}^p$. Then it holds that $g\in\bigcap_{\sigma\in G_K^p}\sigma G_{K_{i,j}}^p\sigma^{-1}$. Therefore, we obtain a desired sequence of extensions of $K$ by induction on $j$.

Now write $L\overset{\mathrm{def}}{=}\bigcup_{i,j}K_{i,j}$. Then $L$ is a weakly unramified pro-$p$ extension of $K$ such that the residue field of $L$ is perfect. Moreover, the closed subgroup
\begin{equation*}
\bigcap_{\sigma\in G_K^p}\sigma G_L^p\sigma^{-1}=\bigcap_{\sigma\in G_K^p}\bigcap_{i,j}\sigma G_{K_{i,j}}^p\sigma^{-1}
\end{equation*}
of $G_L^p$ contains $g$, hence nontrivial. This completes the proof of Claim \hyperlink{henselfreeclaim}{\ref*{henselfree}.A}.

Now let us write $K'$ for the cyclotomic $\mathbb{Z}_p$-extension of $K$. Since the ramification index of $K'L/L$ is infinite, the inertia subgroup $I\subseteq\Gal(K'L/L)\cong\mathbb{Z}_p$ is infinite, hence open in $\Gal(K'L/L)$. Thus, if we write $M$ for the finite Galois extension of $L$ corresponding to $I\subseteq\Gal(K'L/L)$, then $K'L/M$ is a totally ramified $\mathbb{Z}_p$-extension. Thus, it follows from the theory of fields of norms [cf.\ \cite{FV}, Chapter III, \S 5, Exercises, 1; \cite{FV}, Chapter III, (5.5), Theorem; \cite{FV}, Chapter III, (5.7), Theorem] that $G_{K'L}$ is isomorphic to $G_{k_M((t))}$, where $k_M$ is the residue field of $M$. Thus, $G_{K'L}^p$ is a free pro-$p$ group. Since $G_{K'}^p$ is a normal closed subgroup of $G_K^p$ and $G_L^p$ has a nontrivial closed subgroup that is normal in $G_K^p$, it follows from Proposition \ref{snintersect}; Theorem \ref{sneanab}, that $G_{K'L}^p=G_{K'}^p\cap G_L^p$ has a nontrivial normal closed subgroup of $G_K^p$. This completes the proof of Proposition \ref{henselfree}.
\end{proof}

\begin{defi}\label{symquot}
Let $G$ be a profinite group; $Q$ a quotient of $G$ [in the category of profinite groups]; $l$ a prime number.
\begin{enumerate}
\item Write $G_{\mathrm{sym}}$ for the intersection of all normal open subgroups $U\subseteq G$ such that $G/U$ is isomorphic to a symmetric group.
\item Let $U,N\subseteq G$ normal open subgroups of $G$ such that $N$ is contained in $U$. If the kernel of the surjection $G\twoheadrightarrow Q$ coincides with the kernel of the natural surjection $N\cap U_{\mathrm{sym}}\twoheadrightarrow(N\cap U_{\mathrm{sym}})^l$, then we shall say that $Q$ is an \textit{almost symmetric pro-$l$-maximal quotient} of $G$ associated to $(U,N)$.
\item We shall say that $Q$ is an \textit{almost symmetric pro-$l$-maximal quotient} of $G$ if it is the almost symmetric pro-$l$-maximal quotient of $G$ associated to $(U,N)$ for some $U,N\subseteq G$ as in (ii).
\end{enumerate}
\end{defi}

\begin{rem}\label{symquotrem}
It is clear that $G_{\mathrm{sym}}$ is characteristic in $G$. In particular, for $U,N\subseteq G$ as in Definition \ref{symquot}, (ii), the closed subgroup $U_{\mathrm{sym}}$, hence also $N\cap U_{\mathrm{sym}}$, is normal in $G$. Moreover, the kernel of the natural surjection $N\cap U_{\mathrm{sym}}\twoheadrightarrow(N\cap U_{\mathrm{sym}})^l$ is normal in $G$, since it is characteristic in $N\cap U_{\mathrm{sym}}$.
\end{rem}

\begin{prop}\label{hilbsym}
Let $K$ be a Hilbertian field; $Q$ an almost symmetric pro-$p$-maximal quotient of $G_K$. Then $Q$ is strongly sn-internally indecomposable, very sn-elastic, and $Q$ has a nontrivial free pro-$p$ normal closed subgroup.
\end{prop}

\begin{proof}
Every open subgroup $W\subseteq Q$ of $Q$ is isomorphic to the inverse limit of some family of almost pro-$p$-maximal quotients of the inverse image of $W$ in $G_K$. Thus, it follows from Proposition \ref{projlim}; Theorem \ref{sneanab}; \cite{FJ}, Corollaries 12.2.3, 16.3.6, that $Q$ is strongly sn-internally indecomposable and very sn-elastic. Moreover, there exist open subgroups $N,U\subseteq G_K$ of $G_K$ such that $N\subseteq U$ and that $Q$ contains a normal closed subgroup isomorphic to $(N\cap U_{\mathrm{sym}})^p$. Thus, it suffices to show that $(N\cap U_{\mathrm{sym}})^p$ is a nontrivial free pro-$p$ group, i.e., $\cdim_p(N\cap U_{\mathrm{sym}})^p=1$.

It follows from \cite{FJ}, Corollary 12.2.3; \cite{FJ}, Theorem 18.10.4, that the Galois extension of $K$ associated to $U_{\mathrm{sym}}$ is a Hilbertian pseudo algebraically closed field. In particular, it follows from \cite{FJ}, Corollary 11.6.8, that $\cdim_p(N\cap U_{\mathrm{sym}})\le\cdim_p U_{\mathrm{sym}}\le 1$. Thus, it follows from \cite{Se}, Chapter II, \S 2, Proposition 2, that $\cdim_p(N\cap U_{\mathrm{sym}})^p\le 1$. On the other hand, since the Galois extension of $K$ associated to $N\cap U_{\mathrm{sym}}$ is Hilbertian [cf.\ \cite{FJ}, Corollary 12.2.3], $N\cap U_{\mathrm{sym}}$ has a normal open subgroup of index $p$, which implies that $\cdim_p(N\cap U_{\mathrm{sym}})^p\ge 1$. This completes the proof of Proposition \ref{hilbsym}.
\end{proof}

\begin{rem}\label{freesubrem}
At the time of writing of the present paper, the authors do not know whether any almost pro-$p$-maximal quotient of $G_K$ has a nontrivial free pro-$p$ normal closed subgroup or not, where $K$ is a Hilbertian field.
\end{rem}

\begin{thm}\label{inngal}
Let $G$ be a profinite group; $\mathcal{C}$ a full-formation such that $\mathbb{Z}/p\mathbb{Z}$ is a $\mathcal{C}$-group. Suppose that $G$ is isomorphic to one of the following:
\renewcommand{\labelenumi}{(\alph{enumi})}
\begin{enumerate}
\item an almost pro-$\mathcal{C}$-maximal quotient of $G_K$, or an almost symmetric pro-$p$-maximal quotient of $G_K$, where $K$ is a number field or a Henselian discrete valuation field of residue characteristic $p$;
\item $G_K$ or an almost symmetric pro-$p$-maximal quotient of $G_K$, where $K$ is a Hilbertian field.
\end{enumerate}
Then the following hold:
\renewcommand{\labelenumi}{(\roman{enumi})}
\begin{enumerate}
\item $G$ is strongly sn-internally indecomposable.
\item Let $H,H'\subseteq G$ be closed subgroups. Suppose that $H,H'$ have nontrivial closed subgroups that are normal in $G$. Then every families preserving isomorphism $H\overset{\sim}{\to}H'$ in $G$ is induced by an inner automorphism of $G$.
\end{enumerate}
\end{thm}

\begin{proof}
In the case of (a) (respectively, (b)), for a normal open subgroup $N\subseteq G_K$ of $G_K$, we shall write $Q_N$ for the almost pro-$p$-maximal quotient (respectively, almost symmetric pro-$p$-maximal quotient) of $G_K$ associated to $N$ (respectively, associated to $(G_K,N)$).

First, we verify assertion (i). There exists a normal open subgroup $N_0\subseteq G_K$ of $G_K$ such that the natural surjection $G_K\twoheadrightarrow Q_{N_0}$ factors through $G_K\twoheadrightarrow G$. Write $A$ for the set of normal open subgroups of $G_K$ contained in $N_0$. Then, since $G\cong\varprojlim_{N\in A}Q_N$, it follows from Propositions \ref{projlim}, \ref{hilbsym}; Theorem \ref{sneanab}, that $G$ is strongly sn-internally indecomposable. This completes the proof of assertion (i).

Finally, we verify assertion (ii). Let us fix nontrivial normal closed subgroups $N_1,N_1'$ of $G$ contained in $H,H'$, respectively. Then it follows from assertion (i); Proposition \ref{snintersect}, (ii), that $N_1\cap N_1'$ is nontrivial. Write $B$ for the subset of $A$ consisting of $N\in A$ such that $N_1\cap N_1'\not\subseteq N$. Then it holds that $G\cong\varprojlim_{N\in B}Q_N$. Moreover, for $N\in B$, it follows from Propositions \ref{nffree}, \ref{henselfree}, \ref{hilbsym}, that $Q_N$ has a nontrivial free pro-$p$ normal closed subgroup $F_N$. Now it follows from assertion (i); Proposition \ref{snintersect}, together with the definition of $B$, that the intersection of $F_N$ and the image of $N_1\cap N_1'$ in $Q_N$ is a nontrivial free pro-$p$ normal closed subgroup of $Q_N$. Thus, it follows from assertion (i); Theorem \ref{innisom2}; Remark \ref{snfprem}, (iii), that every families preserving isomorphism $H\overset{\sim}{\to}H'$ in $G$ is induced by an inner automorphism of $G$. This completes the proof of assertion (ii), hence also of Theorem \ref{inngal}.
\end{proof}

\begin{rem}\label{inngalrem}
\mbox{}
\begin{enumerate}
\item If $H$ is an open subgroup of an infinite profinite group $G$, then, since $\bigcap_{g\in G}gHg^{-1}$ is a normal open subgroup of $G$, $H$ have a nontrivial closed subgroup that is normal in $G$.
\item Let us consider the case where the supposition of the existence of nontrivial normal closed subgroups in assertion (ii) is replaced by the existence of $n$-subnormal closed subgroups. If we further suppose that, for each $N\in B$, the isomorphism between the images of $H$ and $H'$ in $Q_N$ determined by $\alpha$ is $n$-subnormal in $Q_N$ [where $Q_N$ is the same as in the proof above, and $B$ is determined similarly for a fixed nontrivial $n$-subnormal closed subgroups], then the same proof works. In the case where $G$ itself contains a nontrivial free pro-$p$ normal closed subgroup, it is sufficient to suppose only $\alpha$ is $n$-subnormal in $G$.
\end{enumerate}
\end{rem}

\begin{thm}\label{innsnhilb}
Let $\mathcal{C}$ a nontrivial full-formation; $K$ a Hilbertian field; $S\subseteq G_K$ a nontrivial subnormal closed subgroup of $G_K$; $U,N\subsetneq S$ proper normal open subgroups of $S$ such that $N$ is contained in $U$. Write $Q$ for the almost symmetric pro-$p$-maximal quotient of $S$ associated to $(U,N)$. Then the following hold:
\begin{enumerate}
\item Let $T$ be a profinite group that is isomorphic to an almost pro-$\mathcal{C}$-maximal quotient of $S$ or an almost symmetric pro-$p$-maximal quotient of $S$. If $T$ is nontrivial, then $T$ is strongly sn-internally indecomposable and very sn-elastic.
\item $Q$ has a nontrivial free pro-$p$ normal closed subgroup.
\item Let us denote $G$ by $S$ or $Q$. Let $H,H'\subseteq G$ be closed subgroups. Suppose that $H,H'$ have nontrivial closed subgroups that are normal in $G$. Then every families preserving isomorphism $H\overset{\sim}{\to}H'$ in $G$ is induced by an inner automorphism of $G$.
\end{enumerate}
\end{thm}

\begin{proof}
First, we verify assertion (i). Let $V\subsetneq T$ be a proper open subgroup of $T$. Then it follows from Proposition \ref{hilb1} that the separable extension of $K$ associated to [the inverse image of] $V$ [in $S\subseteq G_K$] is Hilbertian, which implies that $V$ is strongly sn-internally indecomposable and very sn-elastic [cf.\ Theorem \ref{sneanab}]. In particular, every subnormal closed subgroup of $V$ is slim. In light of \cite{MiTs1}, Lemma 1.3, every proper open subgroup of $T$, hence also $T$, does not have nontrivial finite subnormal subgroup. Thus, assertion (i) follows from Proposition \ref{snopen}.

Assertion (ii) follows from an argument similar to the argument in the proof of Proposition \ref{hilbsym}, together with assertion (i); Proposition \ref{hilb1}. Moreover, since $S\cong\varprojlim_{M}Q_M$, where $M$ runs over all normal open subgroups of $S$ contained in $U$, and $Q_M$ is the almost symmetric pro-$p$-maximal quotient of $S$ associated to $(U,M)$, assertion (iii) follows from an argument similar to the argument in the proof of Theorem \ref{inngal}, (ii), together with assertions (i), (ii). This completes the proof of Theorem \ref{innsnhilb}.
\end{proof}

\begin{defi}\label{ow}
Let $G$ be a profinite group; $H,H'\subseteq G$ closed subgroups; $\alpha:H\overset{\sim}{\to}H'$ a continuous isomorphism. Then we shall say that $\alpha$ is \textit{open-wise inner} [in $G$] if for every open subgroup $U\subseteq H$ of $H$, there exists $g\in G$ such that $\alpha(U)=gUg^{-1}$.
\end{defi}

\begin{rem}\label{owrem}
If $\alpha$ is an open-wise inner automorphism in $G$, then so is $\alpha^{-1}$.
\end{rem}

\begin{prop}\label{subinn}
Let $G$ be a profinite group; $H,H'\subseteq G$ closed subgroups; $\alpha:H\overset{\sim}{\to}H'$ a continuous isomorphism; $\{G_i\}_{i\in I}$ a directed subset of the set of closed subgroups of $H\cap H'$ that is normal in $G$ [where $j\ge i\Leftrightarrow G_j\subseteq G_i$] such that the natural homomorphism $G\to\varprojlim_{i\in I}G/G_i$ is an isomorphism. Then the following conditions are equivalent:
\renewcommand{\labelenumi}{(\arabic{enumi})}
\begin{enumerate}
\item $\alpha$ is open-wise inner in $G$.
\item For each $i\in I$, it holds that $\alpha(G_i)=G_i$, and, moreover, the isomorphism $H/G_i\overset{\sim}{\to}H'/G_i$ determined by $\alpha$ is open-wise inner in $G/G_i$.
\item For every closed subgroup $A\subseteq H$, there exists $g\in G$ such that $\alpha(A)=gAg^{-1}$.
\end{enumerate}
In particular, an open-wise inner isomorphism in $G$ is families preserving in $G$.
\renewcommand{\labelenumi}{(\roman{enumi})}
\end{prop}

\begin{proof}
The implication $(3)\Rightarrow(2)$ is clear. First, we verify the implication $(2)\Rightarrow(1)$. Suppose that condition $(2)$ is satisfied. Let $U\subseteq H$ be an open subgroup of $H$. For each $i\in I$, we shall write
\begin{equation*}
E_i\overset{\mathrm{def}}{=}\{g\in G\,\vert\, \alpha(U)\subseteq gUg^{-1}G_i\}\subseteq G.
\end{equation*}
Then, since the image of $U$ via the natural surjection $G\twoheadrightarrow G/G_i$ is an open subgroup of $H/G_i$, $E_i$ is a nonempty [closed] subset of $G$. Moreover, if $j\ge i$, then it is clear that $E_j\subseteq E_i$. Thus, since $G$ is compact, it holds that $\bigcap_{i\in I}E_i\neq\emptyset$.

Let $g\in\bigcap_{i\in I}E_i$. Then it holds that $\alpha(U)\subseteq\bigcap_{i\in I}gUg^{-1}G_i=gUg^{-1}$. Similarly, there exists $h\in G$ such that $\alpha^{-1}(\alpha(U))\subseteq h\alpha(U)h^{-1}$ [cf.\ Remark \ref{owrem}], which implies that $h^{-1}Uh\subseteq\alpha(U)\subseteq gUg^{-1}$. On the other hand, by considering all finite quotients of $G$, $h^{-1}Uh\subseteq gUg^{-1}$ implies that $h^{-1}Uh=gUg^{-1}$, i.e., $\alpha(U)=gUg^{-1}$. This completes the proof of the implication $(2)\Rightarrow(1)$.

Finally, we verify the implication $(1)\Rightarrow(3)$. Suppose that $\alpha$ is open-wise inner in $G$. Let $A\subseteq H$ be a closed subgroup. For each open subgroup $U\subseteq H$ of $H$ containing $A$, we shall write
\begin{equation*}
F_U\overset{\mathrm{def}}{=}\{g\in G\,\vert\, \alpha(A)\subseteq gUg^{-1}\}\subseteq G.
\end{equation*}
Then, since $\alpha$ is open-wise inner, $F_U$ is a nonempty [closed] subset of $G$. Moreover, if $U'\subseteq U$, then, it is clear that $F_{U'}\subseteq F_U$. Thus, since $G$ is compact, it holds that $\bigcap_U F_U\neq\emptyset$.

Let $g\in\bigcap_U F_U$. Then it holds that $\alpha(A)\subseteq\bigcap_U gUg^{-1}=g\left(\bigcap_U U\right)g^{-1}=gAg^{-1}$. In light of Remark \ref{owrem}, it follows from an argument similar to the argument in the proof of implication $(2)\Rightarrow(1)$, we conclude that $\alpha(A)=gAg^{-1}$. This completes the proof of the implication $(1)\Rightarrow(3)$, hence also of Proposition \ref{subinn}.
\end{proof}

\begin{prop}\label{weak-NU}
Let $K$ be a field; $Q$ a quotient of $G_K$ [in the category of profinite groups]; $H,H'\subseteq Q$ closed subgroups of $Q$. Suppose that for any open subgroups $U,V$ of $H, H'$, respectively, if $U$ is isomorphic to $V$, then the field corresponding to [the inverse image of] $U$ [in $G_K$ via the quotient $G_K\twoheadrightarrow Q$] is isomorphic to that of $V$. Then every isomorphism from $H$ to $H'$ is open-wise inner in $Q$.
\end{prop}

\begin{proof}
Let $U\subseteq H$ be an open subgroup of $H$; $\alpha:H\overset{\sim}{\to} H'$ an isomorphism. Then, since $U$ is isomorphic to $\alpha(U)$, the field corresponding to $U$ is isomorphic to that of $\alpha(U)$, which implies that these two fields are conjugate to each other. This completes the proof of Proposition \ref{weak-NU}.
\end{proof}

\begin{rem}\label{uchida}
Theorem \ref{inngal}; Propositions \ref{subinn}, \ref{weak-NU} yield the equivalence of \cite{U1}, Theorem and \cite{U1}, Corollary 2 [see also \cite{I}, Main theorem].

As such, if we obtain a weak version of Neukirch-Uchida type result for a certain quotient of the absolute Galois groups for a class of fields, in many cases we can show the innerness of isomorphisms between open subgroups.
\end{rem}

Next, we consider groups such as free pro-$\mathcal{C}$ groups and surface groups, as well as profinite groups obtained by successive extensions of them. As applications, we obtain results on the fundamental groups of hyperbolic curves over number fields or $p$-adic local fields [cf.\ Theorem \ref{curve}], and on the fundamental group of hyperbolic polycurves [cf.\ Definition \ref{hyppolydef}] over algebraically closed fields of characteristic $0$ [cf.\ Theorem \ref{poly}].

\begin{thm}\label{innanab1}
Let $G$ be a profinite group; $\mathcal{C}$ a full-formation such that $\mathbb{Z}/p\mathbb{Z}$ is a $\mathcal{C}$-group. Suppose that $G$ is isomorphic to one of the following:
\renewcommand{\labelenumi}{(\alph{enumi})}
\begin{enumerate}
\item an almost pro-$\mathcal{C}$-maximal quotient of a free profinite group of [possibly infinite] rank $\ge 2$, or an almost symmetric pro-$p$-maximal quotient of a free profinite group of rank $\ge 2$;
\item an almost pro-$\mathcal{C}$ surface group, or an almost symmetric pro-$p$-maximal quotient of a profinite surface group;
\item a pro-$p$ Demu\v{s}kin group of rank $\ge 3$.
\end{enumerate}
Then the following hold:
\renewcommand{\labelenumi}{(\roman{enumi})}
\begin{enumerate}
\item For each open subgroup $U\subseteq G$ of $G$, the closure of the commutator subgroup of $U^p$ is a free pro-$p$ group of infinite rank.
\item $G$ is strongly sn-internally indecomposable.
\item Let $H,H'\subseteq G$ be closed subgroups. Suppose that $H,H'$ have nontrivial closed subgroups that are normal in $G$. Then every families preserving isomorphism $H\overset{\sim}{\to}H'$ in $G$ is induced by an inner automorphism of $G$.
\end{enumerate}
\end{thm}

\begin{proof}
Assertion (i) follows from Theorem \ref{sneanab}; Proposition \ref{Demushkin1}; \cite{RZ}, Corollary 7.7.5; the discussion preceding \cite{NSW}, Theorem 3.9.11 [observe that every pro-$p$ surface group is a free pro-$p$ group of rank $\ge 2$ or a pro-$p$ Demu\v{s}kin group of rank $\ge 4$]. Assertion (ii) follows from Theorem \ref{sneanab}, Proposition \ref{projlim}. Assertion (iii) follows from an argument similar to the argument in the proof of Theorem \ref{inngal}, (ii), together with assertions (i), (ii).
\end{proof}

\begin{lem}\label{extinnlem}
Let $1\to N\to G\to Q\to 1$ be an exact sequence of profinite groups such that $N$ is center-free; $\alpha\in\Aut(G)$ a continuous automorphism. We regard $N$ as a closed subgroup of $G$ via the inclusion $N\hookrightarrow G$. Write $\rho$ for the outer representation $Q\rightarrow\Out(N)$ associated to the above exact sequence. Suppose that $\alpha$ preserves the subgroup $N\subseteq G$, and that the restriction $\alpha_N\overset{\mathrm{def}}{=}\alpha|_N\in\Aut(N)$ is induced by an inner automorphism of $G$. Write $\alpha_Q\in\Aut(Q)$ for the automorphism of $Q$ induced by $\alpha$. Suppose further that one of the following holds:
\begin{itemize}
\item $\rho$ is injective;
\item $\alpha_Q$ is an inner automorphism of $Q$ and $\Image\rho$ is center-free.
\end{itemize}
Then $\alpha$ is an inner automorphism of $G$.
\end{lem}

\begin{proof}
Write $\phi$ for the [surjective] homomorphism $G\to Q$ in the exact sequence. By replacing $\alpha$ by the composite of $\alpha$ with a suitable inner automorphism of $G$, we may assume that $\alpha_N$ is the identity automorphism of $N$. Then it is clear that $\alpha_Q$ fits into the commutative diagram
\begin{equation*}
\xymatrix@M=5pt{
Q \ar[rd]_{\rho} \ar[rr]^{\alpha_Q} & & Q \ar[ld]^{\rho}\\
 & \rho(Q). & \\
}
\end{equation*}
In particular, if $\rho$ is injective, then $\alpha_Q$ is the identity automorphism of $Q$. In this case, for $g\in G$ and $h\in N$, since $ghg^{-1}\in N$, it holds that $ghg^{-1}=\alpha(ghg^{-1})=\alpha(g)h\alpha(g)^{-1}$, which implies that $g^{-1}\alpha(g)\in Z_G(N)$. On the other hand, since $\phi(\alpha(g))=\alpha_Q(\phi(g))=\phi(g)$, it holds that $g^{-1}\alpha(g)\in\ker\phi=N$. Thus, we conclude that $g^{-1}\alpha(g)\in Z(N)=\{1\}$, i.e., $\alpha$ is the identity automorphism. This completes the proof of Lemma \ref{extinnlem} in the case where $\rho$ is injective.

In the remainder of the proof of Lemma \ref{extinnlem}, we assume that $\alpha_Q$ is an inner automorphism of $Q$ and $\Image\rho$ is center-free.  Let us fix $c\in G$ such that $\alpha_Q$ is the inner automorphism determined by $\phi(c)$. Then, for any $g\in Q$, since $\alpha_Q(g)=\phi(c)g\phi(c)^{-1}$, it follows from the above diagram that $\rho(g)=\rho(\phi(c)g\phi(c)^{-1})$. Thus, it holds that $\rho(\phi(c))\in Z(\rho(Q))=\{1\}$, i.e., $\phi(c)\in\ker\rho$. This implies that $cN\cap Z_G(N)\neq\emptyset$.

Let us fix $d\in cN\cap Z_G(N)$ and write $\beta\in\Aut(G)$ for the automorphism of $G$ determined by $\beta(g)=d^{-1}\alpha(g)d$. Then the automorphisms of $N$ and $Q$ determined by $\beta$ are the identity automorphisms. Thus, by an argument similar to the above argument, we conclude that $\beta$ is the identity automorphism of $G$. This completes the proof of Lemma \ref{extinnlem}.
\end{proof}

\begin{thm}\label{extinn1}
Let $1\to N\to G\to Q\to 1$ be an exact sequence of profinite groups; $\alpha\in\Aut(G)$ a continuous automorphism; $\{G_i\}_{i\in I}$ a directed subset of the set of normal closed subgroups of $G$ [where $j\ge i\Leftrightarrow G_j\subseteq G_i$] such that the natural homomorphism $G\to\varprojlim_{i\in I}G/G_i$ is an isomorphism. We regard $N$ as a closed subgroup of $G$ via the inclusion $N\hookrightarrow G$. Write $\rho$ for the outer representation $Q\rightarrow\Out(N)$ associated to the above exact sequence. Suppose that for each $i\in I$,
\begin{itemize}
\item $\alpha$ preserves the subgroup $N\subseteq G$, and that the restriction of $\alpha$ to $N$ is families preserving in $G$;
\item $N/(N\cap G_i)$ is internally indecomposable;
\item $N/(N\cap G_i)$ contains a nontrivial free pro-$p_i$ normal closed subgroup $F_i$ of $G/G_i$ for some prime number $p_i$.
\end{itemize}
Write $\alpha_Q\in\Aut(Q)$ for the automorphism of $Q$ induced by $\alpha$. Suppose further that one of the following holds:
\begin{itemize}
\item $\rho$ is injective;
\item $\alpha_Q$ is an inner automorphism of $Q$ and $\Image\rho$ is center-free.
\end{itemize}
Then $\alpha$ is an inner automorphism of $G$.
\end{thm}

\begin{proof}
Write $\alpha_i\in\Aut(N/N\cap G_i)$ for the automorphism of $N/N\cap G_i$ induced by $\alpha$ [cf.\ Remark \ref{snfprem}, (ii)]. Since the natural homomorphism $N\to\varprojlim_{i\in I}N/(N\cap G_i)$ is an isomorphism, it follows from Proposition \ref{projlim} that $N$ is internally indecomposable, hence center-free. Thus, in light of Lemmas \ref{innlim}, \ref{extinnlem}, it suffices to show that $\alpha_i$ is induced by an inner automorphism of $G/G_i$ for each $i\in I$.

Since $N/(N\cap G_i)$ is internally indecomposable, the [nontrivial] free pro-$p_i$ normal closed subgroup $F_i$ is center-free, hence of rank $\ge 2$. Thus, it follows from Theorem \ref{autfree}; Remma \ref{snfprem}, (iii), that the restriction of $\alpha_i$ to $F_i$ is induced by an inner automorphism of $G/G_i$. By Proposition \ref{snisom}, $\alpha_i$ is induced by an inner automorphism of $G/G_i$. This completes the proof of Theorem \ref{extinn1}.
\end{proof}

\begin{cor}\label{extinn2}
Let $1\to N\to G\to Q\to 1$ be an exact sequence of profinite groups; $\alpha\in\Aut(G)$ a continuous automorphism; $\mathcal{C}$ a full-formation such that $\mathbb{Z}/p\mathbb{Z}$ is a $\mathcal{C}$-group. We regard $N$ as a closed subgroup of $G$ via the inclusion $N\hookrightarrow G$. Write $\rho$ for the outer representation $Q\rightarrow\Out(N)$ associated to the above exact sequence. Suppose that $\alpha$ preserves the subgroup $N\subseteq G$, and that the restriction of $\alpha$ to $N$ is families preserving in $G$. Write $\alpha_Q\in\Aut(Q)$ for the automorphism of $Q$ induced by $\alpha$. Suppose further that one of the following holds:
\begin{itemize}
\item $\rho$ is injective;
\item $\alpha_Q$ is an inner automorphism of $Q$ and $\Image\rho$ is center-free.
\end{itemize}
Furthermore, suppose that $N$ is isomorphic to one of the following:
\renewcommand{\labelenumi}{(\alph{enumi})}
\begin{enumerate}
\item an almost pro-$\mathcal{C}$-maximal quotient of a free profinite group of [possibly infinite] rank $\ge 2$, or an almost symmetric pro-$p$-maximal quotient of a free profinite group of rank $\ge 2$;
\item an almost pro-$\mathcal{C}$ surface group, or an almost symmetric pro-$p$-maximal quotient of a profinite surface group;
\item a pro-$p$ Demu\v{s}kin group of rank $\ge 3$;
\item an almost pro-$\mathcal{C}$-maximal quotient of $G_K$, or an almost symmetric pro-$p$-maximal quotient of $G_K$, where $K$ is a Henselian discrete valuation field of characteristic $p$.
\end{enumerate}
\renewcommand{\labelenumi}{(\roman{enumi})}
Then $\alpha$ is an inner automorphism of $G$.
\end{cor}

\begin{proof}
For an open subgroup $U\subseteq N$ of $N$ that is normal in $G$, it follows from Theorem \ref{innanab1}, (i); the proof of Proposition \ref{henselfree}, that $U^p$ has a characteristic closed subgroup $F_U$ that is free pro-$p$ of rank $\ge 2$. Since $U^p$ is the image of $U\subseteq G$ via the natural surjection $G\twoheadrightarrow G/\ker(U\twoheadrightarrow U^p)$, it holds that $U^p$, hence also $F_U$, is normal in $G/\ker(U\twoheadrightarrow U^p)$. Moreover, it follows from Theorem \ref{inngal}, (i); Theorem \ref{innanab1}, (ii), that the almost pro-$p$-maximal quotient of $N$ associated to $U$ is internally indecomposable. Thus, it follows from Theorem \ref{extinn1} that $\alpha$ is an inner automorphism of $G$. This completes the proof of Corollary \ref{extinn2}.
\end{proof}

\begin{thm}\label{extinn3}
Let $n$ be a positive integer; $G_1,\ldots,G_{n+1}$ profinite groups. Suppose that
\begin{itemize}
\item for each $i\in\{1,\ldots,n\}$, there exists a surjective homomorphism $\phi_i:G_i\twoheadrightarrow G_{i+1}$ such that $\ker\phi_i$ is isomorphic to one of (a)-(d) in Corollary \ref{extinn2}, and, moreover, the outer representation $G_{i+1}\to\Out(\ker\phi_i)$ is injective or has a center-free image;
\item if the outer representation $G_{n+1}\to\Out(\ker\phi_n)$ is not injective, then every families preserving automorphism [in $G_{n+1}$] of $G_{n+1}$ is an inner automorphism.
\end{itemize}
Then every families preserving automorphism [in $G_1$] of $G_1$ is an inner automorphism.
\end{thm}

\begin{proof}
Let $\alpha\in\Aut(G_1)$ be a families preserving automorphism. Then it follows from Remark \ref{snfprem}, (iii), that, for each $i\in\{1,\ldots,n+1\}$, $\alpha$ determines a families preserving automorphism $\alpha_i\in\Aut(G_i)$. In particular, if the outer representation $G_{n+1}\to\Out(\ker\phi_n)$ is not injective, then $\alpha_{n+1}$ is an inner automorphism. Moreover, for each $i\in\{1,\ldots,n\}$, the automorphism of $\ker\phi_i$ determined by $\alpha_i$ is families preserving in $G_i$. Thus, by applying Corollary \ref{extinn2} inductively, we conclude that $\alpha_i$ is an inner automorphism of $G_i$ for each $i\in\{1,\ldots,n\}$. This completes the proof of Theorem \ref{extinn3}.
\end{proof}

\begin{thm}\label{curve}
Let $K$ be an algebraic extension of a number field or its completion with respect to a nontrivial valuation; $X$ a hyperbolic curve over $K$. Then every families preserving automorphism of $\pi_1(X)$ is an inner automorphism.
\end{thm}

\begin{proof}
It follows from \cite{NodNon}, Theorem C, that the outer representation $G_K\to\Out(\pi_1(X\times_K \overline{K}))$ is injective. Since $\pi_1(X\times_K \overline{K})$ is a surface group, Theorem \ref{curve} follows from Theorem \ref{extinn3}.
\end{proof}

\begin{defi}[cf.\ \cite{H1}, Definition 2.1, (ii)]\label{hyppolydef}
Let $n$ be a positive integer; $S$ a scheme; $X$ a scheme over $S$. Then we shall say that $X$ is a \textit{hyperbolic polycurve [of relative dimension $n$]} over $S$ if there exists a [not necessarily unique] sequence of schemes
\begin{equation*}
X=X_n\to X_{n-1}\to\cdots\to X_2\to X_1\to X_0=S
\end{equation*}
such that $X_i\to X_{i-1}$ is a hyperbolic curve for each $i\in\{1,\ldots,n\}$. We shall refer to the above sequence as a \textit{sequence of parametrizing morphisms}.
\end{defi}

\begin{thm}\label{poly}
Let $n$ be a positive integer; $K$ an algebraically closed field of characteristic $0$; $X$ a hyperbolic polycurve of relative dimension $n$ over $K$;
\begin{equation*}
X=X_n\to X_{n-1}\to\cdots\to X_2\to X_1\to X_0=\Spec K
\end{equation*}
a sequence of parametrizing morphisms (respectively, a sequence of parametrizing morphisms satisfying condition $(*)_p$ defined in \cite{Sa1}, Definition 3.10). Suppose that, for each $i\in\{1,\ldots,n\}$, the image of the outer representation $\pi_1(X_{i-1})\to\Out(\ker(\pi_1(X_i)\to\pi_1(X_{i-1})))$ (respectively, $\pi_1(X_{i-1})^p\to\Out(\ker(\pi_1(X_i)^p\to\pi_1(X_{i-1})^p))$) is center-free [e.g.\ the case where the outer representation is injective]. Then every families preserving automorphism of $\pi_1(X)$ is an inner automorphism.
\end{thm}

\begin{proof}
Since $\ker(\pi_1(X_i)\to\pi_1(X_{i-1}))$ (respectively, $\ker(\pi_1(X_i)^p\to\pi_1(X_{i-1})^p)$) is a surface group (respectively, pro-$p$ surface group) [cf.\ \cite{H1}, Proposition 2.4, (i); \cite{Sa1}, Definition 3.10], Theorem \ref{poly} follows from Theorem \ref{extinn3}.
\end{proof}

\begin{rem}\label{polyrem}
If $X$ is the $n$-th configuration space associated to a hyperbolic curve over $K$ [cf.\ \cite{MoTa}, Definition 2.1], then, with respect to the sequence determined by projection morphisms, $X$ satisfies condition $(*)_p$ defined in \cite{Sa1}, Definition 3.10 [cf.\ \cite{MoTa}, Proposition 2.2, (iii)], and, moreover, the outer representations appearing in Theorem \ref{poly} is injective [cf.\ \cite{As}, Theorem 1; \cite{As}, Remark following the proof of Theorem 1].
\end{rem}

At the end of the present section, we show that similar results to the case of profinite groups hold for cases such as topological fundamental groups.

\begin{lem}\label{snfplim}
Let $n$ be a positive integer; $G$ a group; $H,H'\subseteq G$ subgroups; $\alpha:H\overset{\sim}{\to}H'$ an isomorphism; $\{G_i\}_{i\in I}$ a directed subset of the set of normal subgroups of $G$ of finite index [where $j\ge i\Leftrightarrow G_j\subseteq G_i$]. Suppose that for each $i\in I$, it holds that $\alpha(H\cap G_i)=H'\cap G_i$, and, moreover, the isomorphism $\alpha_i:H/(H\cap G_i)\overset{\sim}{\to}H'/(H'\cap G_i)$ of finite groups determined by $\alpha$ is families preserving (respectively, $n$-subnormal) in $G/G_i$. Then the isomorphism $\varprojlim_{i\in I} H/(H\cap G_i)\overset{\sim}{\to}\varprojlim_{i\in I} H'/(H'\cap G_i)$ between [closed] subgroups of [the profinite group] $\varprojlim_{i\in I} G/G_i$ is families preserving (respectively, $n$-subnormal) in $\varprojlim_{i\in I} G/G_i$.
\end{lem}

\begin{proof}
Write $\tilde{G}\overset{\mathrm{def}}{=}\varprojlim_{i\in I} G/G_i$, and, moreover, for each $i\in I$, write $p_i:\tilde{G}\twoheadrightarrow G/G_i$ for the natural surjective homomorphism [cf.\ \cite{RZ}, Proposition 1.1.10].

First, we verify the families preserving case. Let $J\subseteq\tilde{G}$ be a pro-cyclic subgroup of $\tilde{G}$, and for each $i\in I$, write
\begin{equation*}
E_i\overset{\mathrm{def}}{=}\{g\in\tilde{G}\,\vert\, \alpha_i(p_i(J))=p_i(g)p_i(J)p_i(g)^{-1}\}\subseteq\tilde{G}.
\end{equation*}
Then, since $\alpha_i$ is families preserving in $G/G_i$, $E_i$ is a nonempty [closed] subset of $\tilde{G}$. Moreover, if $j\ge i$, then it holds that $E_j\subseteq E_i$. Thus, since $\tilde{G}$ is compact, it holds that $\bigcap_{i\in I}E_i\neq\emptyset$.

Let $g\in\bigcap_{i\in I}E_i$. Then, for each $i\in I$, it holds that $p_i(\alpha(J))=\alpha_i(p_i(J))=p_i(gJg^{-1})$. Thus, it follows from \cite{RZ}, Proposition 2.1.4, (a), that $\alpha(J)=gJg^{-1}$. This completes the proof of the families preserving case of Lemma \ref{snfplim}.

Finally, we verify the $n$-subnormal case. Let $S\subseteq\tilde{G}$ be an $n$-subnormal closed subgroup of $\tilde{G}$ contained in both $\varprojlim_{i\in I} H/(H\cap G_i)$ and $\varprojlim_{i\in I} H'/(H'\cap G_i)$. Then, for each $i\in I$, since $p_i(S)$ is an $n$-subnormal [closed] subgroup of $G/G_i$ contained in both $H/(H\cap G_i)$ and $H'/(H'\cap G_i)$, it holds that $p_i(\alpha(S))=\alpha_i(p_i(S))=p_i(S)$. Thus, it follows from \cite{RZ}, Proposition 2.1.4, (a), that $\alpha(S)=S$. This completes the proof of Lemma \ref{snfplim}.
\end{proof}

\begin{lem}\label{commensurator}
Let $G$ be a group (respectively, topological group) isomorphic to a free group of infinite rank or the topological fundamental group of a hyperbolic curve over $\mathbb{C}$ (respectively, isomorphic to the tempered fundamental group of a hyperbolic curve over $\mathbb{C}_p$). Then $G$ is commensurably terminal in $G^{\wedge}$, i.e., for $g\in G^{\wedge}$, if $gGg^{-1}\cap G$ has finite index in both $G$ and $gGg^{-1}$, then it holds that $g\in G$.
\end{lem}

\begin{proof}
Since the moduli stack of stable curves of given type is proper [cf.\ \cite{DM}, Definition 4.11; \cite{DM}, Theorem 5.2; \cite{Kn}, Theorem 2.7], the valuative criterion implies that a hyperbolic curve over $\mathbb{C}_p$ always has stable reduction. Thus, the resp'd case follows from the proof of \cite{IUTI}, Proposition 2.4, (iii). We verify the non-resp'd case. Let $g\in G^{\wedge}$ be an element such that $gGg^{-1}\cap G$ has finite index in both $G$ and $gGg^{-1}$. Write $H\overset{\mathrm{def}}{=}gGg^{-1}\cap G$. Then, since $H\subseteq gGg^{-1}$, it holds that $g^{-1}Hg\subseteq G$. In particular, if $G$ is isomorphic to the topological fundamental group of a hyperbolic curve over $\mathbb{C}$, then it follows from \cite{IUTI}, Theorem 2.6, that $g^{-1}\in G$, hence $g\in G$.

Now suppose that $G$ is a free group on an infinite set $X$. Let us fix distinct elements $h_1,h_2\in X$. Write $n\overset{\mathrm{def}}{=}[G:H]$. Then, since $g^{-1}h_1^n g,g^{-1}h_2^ng\in g^{-1}Hg\subseteq G$, it follows from \cite{P}, Theorem 3.2, that for $i=1,2$, there exists $k_i\in G$ such that $g^{-1}h_i^n g=k_i^{-1}h_i^nk_i$, which implies that $gk_i^{-1}\in Z_{G^{\wedge}}(h_i^n)$. On the other hand, $G^{\wedge}$ is a free profinite group on a set containing $X$ [cf.\ \cite{RZ}, Exercise 3.3.3; \cite{RZ}, Exercise 3.5.14, (d)]. Thus, it holds that $Z_{G^{\wedge}}(h_i^n)\subseteq G^{\wedge}$ is the pro-cyclic subgroup of $G^{\wedge}$ generated by $h_i$. This implies that there exists $a_i\in\widehat{\mathbb{Z}}$ such that $gk_i^{-1}=h_i^{a_i}$. Now we conclude from $h_1^{-a_1}h_2^{a_2}=k_1g^{-1}\cdot gk_2^{-1}=k_1k_2^{-1}\in G$ that $a_1,a_2\in\mathbb{Z},\ g=h_1^{a_1}k_1\in G$. This completes the proof of Lemma \ref{commensurator}.
\end{proof}

\begin{thm}\label{innsurface}
Let $G$ be a group (respectively, topological group) isomorphic to the topological fundamental group of a hyperbolic curve over $\mathbb{C}$ (respectively, isomorphic to the tempered fundamental group of a hyperbolic curve over $\mathbb{C}_p$); $H,H'\subseteq G$ subgroups (respectively, open subgroups) of finite index; $\alpha:H\overset{\sim}{\to}H'$ an isomorphism (respectively, a continuous isomorphism). 

Suppse that for every cyclic subgroup $I\subseteq H$, there exists $g\in G$ such that $\alpha(I)=gIg^{-1}$. Then $\alpha$ is induced by an inner automorphism of $G$.
\end{thm}

\begin{proof}
Let us observe that, for any normal subgroup $N\subseteq G$ of $G$ of finite index, the isomorphism $H/(H\cap N)\overset{\sim}{\to}H'/(H'\cap N)$ determined by $\alpha$ is families preserving in $G/N$. Since $G$ is residually finite [cf.\ \cite{MoTa}, Proposition 7.1, (ii); \cite{An}, \S 4.5], it follows from Theorem \ref{innanab1}; Lemma \ref{snfplim}; \cite{An}, Proposition 4.4.1, that there exists $g\in G^{\wedge}$ such that for any $h\in H$, it holds that $\alpha(h)=ghg^{-1}$. In particular, it holds that $H'=gHg^{-1}$. Since $H'=gHg^{-1}(\subseteq gGg^{-1}\cap G)$ has finite index in both $G$ and $gGg^{-1}$, we conclude from Lemma \ref{commensurator} that $g\in G$. This completes the proof of Theorem \ref{innsurface}.
\end{proof}

Although it is not the main subject of this paper to study normal automorphisms, if $G$ is free, then, by using a result in \cite{J}, we obtain a result stronger than Theorem \ref{innsurface}:

\begin{thm}\label{innfree}
Let $G$ be a group (respectively, topological group) isomorphic to a free group of [possibly infinite] rank $\ge 2$ (respectively, isomorphic to the tempered fundamental group of an affine hyperbolic curve over $\mathbb{C}_p$); $H,H'\subseteq G$ subgroups (respectively, open subgroups) of finite index; $\alpha:H\overset{\sim}{\to}H'$ an isomorphism (respectively, a continuous isomorphism). 

Suppose that $\alpha(N)=N$ for any normal subgroup (respectively, normal closed subgroup) $N\subseteq G$ of $G$ contained in $H\cap H'$. Then $\alpha$ is induced by an inner automorphism of $G$.
\end{thm}

\begin{proof}
For any normal subgroup $N\subseteq G$ of $G$ of finite index, if $N$ is contained in $H\cap H'$, then the isomorphism $H/N\overset{\sim}{\to}H'/N$ determined by $\alpha$ is normal in $G/N$. Since $G$ is residually finite [cf.\ \cite{RZ}, Proposition 3.3.15; \cite{An}, \S 4.5], Theorem \ref{innfree} follows from an argument similar to the argument in the proof of Theorem \ref{innsurface}, together with \cite{J}, Main Theorem.
\end{proof}

\vskip.5\baselineskip
\section{A generalization of the Neukirch-Uchida theorem to $l$-quasi-number fields}\label{lqnf}
In the present section, in a certain case where a weak version of Neukirch-Uchida type result is not known [and hence Proposition \ref{weak-NU} cannot be applied], from an anabelian geometric point of view, we directly prove families preservingness of isomorphisms between [certain closed subgroups of quotients of] absolute Galois groups. This approach yields a generalization of the Neukirch-Uchida type result to a wider class of fields [cf.\ Theorem \ref{lqnfinn}].

\begin{defi}\label{qnf}
Let $l$ be a prime number; $\Sigma$ a set of prime numbers; $K$ an algebraic extension of $\mathbb{Q}$. Then we shall say that $K$ is an \textit{$l$-quasi-number field} if we write $L\supseteq K$ for the Galois closure of $K$ over $\mathbb{Q}$, then $[L:\mathbb{Q}]$ is not divided by $l^\infty$. We shall say that $K$ is a \textit{$\Sigma$-quasi-number field} if $K$ is a $p$-quasi-number field for every $p\in\Sigma$.
\end{defi}

\begin{prop}\label{qnfext}
Let $\Sigma$ be a set of prime numbers; $K,L$ algebraic extensions of $\mathbb{Q}$ such that $L\supseteq K$. Then the following hold:
\begin{enumerate}
\item If $L$ is a $\Sigma$-quasi-number field, then so is $K$.
\item Let $F$ be a number field contained in $K$. Write $M\supseteq K$ for the Galois closure of $K$ over $F$. Then $K$ is a $\Sigma$-quasi-number field if and only if $[M:F]$ is not divided by $l^\infty$ for every $l\in\Sigma$.
\item If $L/K$ is a finite extension or a pro-prime-to-$\Sigma$ extension, then $K$ is a $\Sigma$-quasi-number field if and only if so is $L$.
\item If $K$ is a $\Sigma$-quasi-number field, then $(G_K^l)^{\mathrm{ab}}$ is infinite for every $l\in\Sigma$.
\end{enumerate}
\end{prop}

\begin{proof}
Assertion (i) is immediate. Assertion (iv) immediately follows by considering the cyclotomic $\mathbb{Z}_l$-extension. First, we verify assertion (ii). Since $M$ is contained in the Galois closure of $K$ over $\mathbb{Q}$, necessity is immediate. We verify sufficiency. Suppose that $[M:F]$ is not divided by $l^\infty$ for every $l\in\Sigma$. Then, since $G_M$ is normal in $G_F$ and $G_F$ is open in $G_{\mathbb{Q}}$, the number of $G_{\mathbb{Q}}$-conjugates of $G_M$ is finite. Thus, if we write $M'$ for the extension of $G_M$ corresponding to $\bigcap_{g\in G_\mathbb{Q}}gG_Mg^{-1}\subseteq G_M$, then $M'$ is Galois over $\mathbb{Q}$, and $[M':\mathbb{Q}]=[M':F]\cdot[F:\mathbb{Q}]$ is not divided by $l^\infty$ for every $l\in\Sigma$. Since $K$ is contained in $M'$, we conclude that $K$ is a $\Sigma$-quasi-number field. This completes the proof of assertion (ii).

Finally, we verify assertion (iii). In light of assertion (i), it suffices to show that if $K$ is a $\Sigma$-quasi-number field, then so is $L$. Suppose that $K$ is a $\Sigma$-quasi-number field. Write $M\supseteq K$ for the Galois closure of $K$ over $\mathbb{Q}$.

If $L/K$ is finite, then it follows from \cite{FJ}, Lemma 1.2.5, (b), that there exists a finite extension $F\supseteq\mathbb{Q}$ of $\mathbb{Q}$ such that $G_L=G_K\cap G_F=G_{KF}$. Then, since $MF\supseteq KF=L$ is Galois over $F$, it follows from assertion (ii) that $L$ is a $\Sigma$-quasi-number field.

If $L/K$ is a pro-prime-to-$\Sigma$ extension, then write $M'\supseteq M$ for the maximal pro-prime-to-$\Sigma$ extension of $M$. Then, since $G_{M'}$ is characteristic in $G_M$, $M'$ is Galois over $\mathbb{Q}$. Moreover, since $M$ is a $\Sigma$-quasi-number field, $M'$ is also a $\Sigma$-quasi-number field. On the other hand, since $LM$ is a pro-prime-to-$\Sigma$ extension of $M$, it holds that $L\subseteq LM\subseteq M'$. Thus, assertion (i) implies that $L$ is a $\Sigma$-quasi-number field. This completes the proof of assertion (iii), hence also of Proposition \ref{qnfext}.
\end{proof}

\begin{cor}\label{qnfextcor}
Let $\Sigma$ be a set of prime numbers; $K_1,K_2$ $\Sigma$-quasi-number fields; $\varphi:G_{K_1}\overset{\sim}{\to}G_{K_2}$ a continuous isomorphism. Write $L_1\supseteq K_1$ for the Galois closure of $K$ over $\mathbb{Q}$; $L_2\supseteq K_2$ for the algebraic extension of $K$ corresponding to $\varphi(G_{L_1})\subseteq G_{K_2}$. Then $L_2$ is a $\Sigma$-quasi-number field.
\end{cor}

\begin{proof}
Let us fix a prime number $l\in\Sigma$. It suffices to show that $L_2$ is an $l$-quasi-number field. Since $K_1$ is an $l$-quasi-number field, $[L_1:\mathbb{Q}]$, hence also $[L_1:K_1]$, is not divided by $l^\infty$. Moreover, since $L_1$ is Galois over $K_1$, $G_{L_2}=\varphi(G_{L_1})\subseteq G_{K_2}$ is normal in $G_{K_2}$. Thus, $L_2$ is Galois over $K_2$, and, moreover, $[G_{K_2}:G_{L_2}]=[G_{K_1}:G_{L_1}]=[L_1:K_1]$ is not divided by $l^\infty$.

Now it follows from \cite{RZ}, Proposition 2.3.2, (b), that there exists an open subgroup $U\subseteq G_{K_2}$ of $G_{K_2}$ containing $G_{L_2}$ such that $[U:G_{L_2}]$ is not divided by $l$. Write $M_2$ for the subextension of $L_2/K_2$ corresponding to $U$, then $L_2/M_2$ is a prime-to-$l$ [Galois] extension and $M_2/K_2$ is a finite extension. Thus, it follows from Proposition \ref{qnfext}, (iii), that $L_2$ is an $l$-quasi-number field. This completes the proof of Corollary \ref{qnfextcor}.
\end{proof}

\begin{prop}\label{decomp}
Let $l$ be a prime number; $K$ an algebraic extension of $\mathbb{Q}$; $H\subseteq G_K$ a closed subgroup of $G_K$. Consider the condition

\begin{tabular}{m{4em} m{24em}}
$(\dagger)_{l,K,H}:$ & There exists an open subgroup $V\subseteq H$ of $H$ such that, for any open subgroup $U\subseteq V$ of $V$, it holds that $\dim_{\mathbb{F}_l}H^2(U,\mathbb{F}_l)=1$, where $\mathbb{F}_l$ is a $U$-module equipped with trivial action.
\end{tabular}

Then the following hold:
\begin{enumerate}
\item If $(\dagger)_{l,K,H}$ is satisfied, then $H$ is a closed subgroup of $D_{K,\overline{v}}$ for some $\overline{v}\in\mathbb{V}(\overline{\mathbb{Q}})^{\mathrm{non}}$.
\item Let $\overline{v}\in\mathbb{V}(\overline{\mathbb{Q}})^{\mathrm{non}}$. Suppose that $[D_{\mathbb{Q},\overline{v}}:D_{K,\overline{v}}]$ is not divided by $l^\infty$, and that $H$ is an open subgroup of $D_{K,\overline{v}}$. Then $(\dagger)_{l,K,H}$ is satisfied.
\item If $K$ is an $l$-quasi-number field, then the set of decomposition subgroups of $G_K$ associated to an element of $\mathbb{V}(\overline{\mathbb{Q}})^{\mathrm{non}}$ coincides with the set of maximal elements [with respect to inclusion] of the set of closed subgroups $H\subseteq G_K$ satisfying $(\dagger)_{p,K,H}$ for some prime number $p$.
\end{enumerate}
\end{prop}

\begin{proof}
First, we verify assertion (i). Suppose that condition $(\dagger)_{l,K,H}$ is satisfied, and let $V$ be as in condition $(\dagger)_{l,K,H}$. Write $L$ for the separable extension of $K$ corresponding to $V$. We may assume that $L$ is totally imaginary and contains $\mu_{l}$. Let $U\subseteq V$ be an open subgroup of $V$. Then it follows from condition $(\dagger)_{l,K,H}$ that $\dim_{\mathbb{F}_l}H^2(U,\mu_l)=\dim_{\mathbb{F}_l}H^2(U,\mathbb{F}_l)=1$ [cf.\ our assumption that $\mu_l\subseteq L$]. Then it follows from an argument similar to the argument in the proof of \cite{NSW}, Theorem 12.1.9, that $V=D_{L,\overline{v}}$ for some $\overline{v}\in\mathbb{V}(\overline{\mathbb{Q}})^{\mathrm{non}}$. Then it follows from \cite{NSW}, Lemma 12.1.10, that $H\subseteq D_{K,\overline{v}}$. This completes the proof of assertion (i).

Next, we verify assertion (ii). Write $p=p_{\overline{v}}$ and identify $D_{\mathbb{Q},\overline{v}}$ with $G_{\mathbb{Q}_p}$. Then there exists an open subgroup $V$ of $H$ such that the algebraic extension of $\mathbb{Q}_p$ corresponding to $V$ contains $\mu_l$. Now let $U\subseteq V$ be an open subgroup. Then the action of $U$ on $\mu_l$ is trivial, i.e., $H^2(U,\mathbb{F}_l)\cong H^2(U,\mu_l)$. On the other hand, since $[G_{\mathbb{Q}_p}:D_{K,\overline{v}}]$, hence also $[G_{\mathbb{Q}_p}:U]$, is not divided by $l^\infty$, it follows from \cite{RZ}, Proposition 2.3.2, (b), that there exists an open subgroup $U_0\subseteq G_{\mathbb{Q}_p}$ of $G_{\mathbb{Q}_p}$ containing $U$ such that $[U_0:U]$ is not divided by $l$. Then it follows from \cite{NSW}, Proposition 1.5.1, that $H^2(U,\mu_l)\cong\varinjlim_W H^2(W,\mu_l)$, where $W$ runs over all open subgroups of $U_0$ containing $U$. Moreover, if $W_1,W_2$ are open subgroups of $U_0$ containing $U$ such that $W_1\supseteq W_2$, then it follows from \cite{NSW}, Corollary 7.1.4; \cite{NSW}, Theorem 7.1.8, (ii), that, $H^2(W_1,\mu_l)\cong H^2(W_2,\mu_l)\cong\mathbb{F}_l$, and the transition map $H^2(W_1,\mu_l)\to H^2(W_2,\mu_l)$ is obtained by multiplication by $[W_1:W_2]$. Thus, since $[W_1:W_2]$ is not divided by $l$, the transition map is isomorphic, which implies that $H^2(U,\mu_l)\cong H^2(U_0,\mu_l)\cong\mathbb{F}_l$. This completes the proof of assertion (ii). Finally, assertion (iii) immediately follows from assertions (i), (ii), together with \cite{NSW}, Corollary 12.1.3. This completes the proof of Proposition \ref{decomp}.
\end{proof}

\begin{rem}\label{decomprem}
Proposition \ref{decomp} is based on \cite{Sa2}, Lemma 3.4, which treats positive characteristic global fields. Note that, as mentioned in \cite{Sa2}, Remark 5, it is essentially given by J.\ Neukirch [cf.\ \cite{N1}, Theorem 1]. Assertion (iii) also follows from \cite{EP}, Theorem 5.4.3.
\end{rem}

\begin{prop}\label{plocprime}
Let $p$ be a prime number; $H\subseteq G_{\mathbb{Q}_p}$ a closed subgroup of $G_{\mathbb{Q}_p}$. Suppose that $H$ contains a nontrivial normal closed subgroup of $G_{\mathbb{Q}_p}$. Then $p$ is the unique prime number $l$ such that $H$ has a nontrivial pro-$l$ normal closed subgroup.
\end{prop}

\begin{proof}
Let $N\subseteq H$ be a nontrivial closed subgroup of $H$ that is normal in $G_{\mathbb{Q}_p}$. Write $P$ for the wild inertia subgroup of $G_{\mathbb{Q}_p}$, which is a nontrivial pro-$p$ normal closed subgroup of $G_{\mathbb{Q}_p}$. Then it follows from Theorem \ref{sneanab}; Proposition \ref{snintersect}, (ii), that $N\cap P\neq\{1\}$.

Now let $l\neq p$ be a prime number; $N'\subseteq H$ a pro-$l$ normal closed subgroup of $H$. Then, since $N\cap P$ is pro-$p$, it holds that $N'\cap (N\cap P)=\{1\}$. Thus, since $N'$ and $N\cap P$ are normal closed subgroup of $H$, it holds that $N'\subseteq Z_H(N\cap P)\subseteq Z_{G_{\mathbb{Q}_p}}(N\cap P)$. Now it follows from Theorem \ref{sneanab} that $Z_{G_{\mathbb{Q}_p}}(N\cap P)=\{1\}$, which implies that $N'=\{1\}$. This completes the proof of Proposition \ref{plocprime}.
\end{proof}

\begin{cor}\label{decompprime}
Let $l$ be a prime number; $K$ an $l$-quasi-number field; $\overline{v}\in\mathbb{V}(\overline{\mathbb{Q}})^{\mathrm{non}}$. Then $p_{\overline{v}}$ is the unique prime number $p$ such that $D_{K,\overline{v}}$ has a nontrivial pro-$p$ normal closed subgroup.
\end{cor}

\begin{proof}
Since $K$ is an $l$-quasi-number field, there exists a closed subgroup $N\subseteq D_{K,\overline{v}}$ of $D_{K,\overline{v}}$ such that $N$ is normal in $D_{\mathbb{Q},\overline{v}}$ and that $[D_{\mathbb{Q},\overline{v}}:N]$ is not divided by $l^\infty$. In particular, $N$ is nontrivial. Thus, Corollary \ref{decompprime} follows from Proposition \ref{plocprime}.
\end{proof}

\begin{prop}\label{qnftchebotarev}
Let $l$ be a prime number; $K$ an $l$-quasi-number field; $L\supseteq K$ a finite Galois $l$-extension. Then for any cyclic subgroup $I\subseteq\Gal(L/K)$, there exist a number field $K_0$ contained in $K$; a finite Galois extension $L_0\supseteq K_0$ of $K_0$ contained in $L$; elements $v\in\mathbb{V}(K)^{\mathrm{non}}$, $w\in\mathbb{V}(L)^{\mathrm{non}}$, $v_0\in\mathbb{V}(K_0)^{\mathrm{non}}$, and $w_0\in\mathbb{V}(L_0)^{\mathrm{non}}$ satisfying the following conditions:
\begin{itemize}
\item $L=KL_0$, and, moreover, the natural homomorphism $\Gal(L/K)\to\Gal(L_0/K_0)$ is an isomorphism;
\item $[K:K_0]$ is not divided by $l$;
\item $w$ is lying over both $v$ and $w_0$;
\item $L/K$ is unramified over $v$;
\item $w_0$ is unramified over $\mathbb{Q}_{w_0}$;
\item $D_{L/K,w}=I$;
\item $v$ is lying over $v_0$, and the order of the inertia subgroup of $v$ over $v_0$ is not divided by $l$ [as a supernatural number].
\end{itemize}
\end{prop}

\begin{proof}
Since $[K:\mathbb{Q}]$ is not divided by $l^\infty$, there exists a number field $F$ contained in $K$ such that $[K:F]$ is not divided by $l$. Then it follows from Lemma \ref{normalize} that there exists an open subgroup $U\subseteq G_F$ of $G_F$ containing $G_K$ and a normal open subgroup $V\subseteq U$ of $U$ such that $G_L=G_K\cap V$ and that $(\Gal(L/K)\cong)G_K/G_L\overset{\sim}{\to}U/V$. Write $K_0$ for the finite separable extension of $F$ corresponding to $U\subseteq G_F$; $L_0$ for the finite Galois extension of $K_0$ corresponding to $V\subseteq U=G_{K_0}$; $f$ for the isomorphism $\Gal(L/K)\overset{\sim}{\to}\Gal(L_0/K_0)$.

Here, it follows from Tchebotarev's density theorem [cf.\ \cite{NSW}, (9.1.3)] that there exist elements $v_0\in\mathbb{V}(K_0)^{\mathrm{non}}$ and $w_0\in\mathbb{V}(L_0)^{\mathrm{non}}$ such that
\begin{itemize}
\item $w_0$ is lying over $v_0$;
\item $L_0/K_0$ is unramified over $v_0$;
\item $w_0$ is unramified over $\mathbb{Q}_{w_0}$;
\item $D_{L_0/K_0,w_0}=f(I)$.
\end{itemize}
Let $v\in\mathbb{V}(K)^{\mathrm{non}}$ be an element lying over $v_0$. It suffices to show that there exists $w\in\mathbb{V}(L)^{\mathrm{non}}$ such that it is lying over both $v$ and $w_0$, and that $D_{L/K,w}=I$.

Now we claim the following:

\begin{quotation}\hypertarget{qnftchebotarevclaim}{}
Claim \ref*{qnftchebotarev}.A: For every finite subextension field $M$ of $K/K_0$, there exists $v_M\in\mathbb{V}(M)^{\mathrm{non}}$ lying over $v_0$ such that the ramification index of $v_M$ over $v_0$ is not divided by $l$, and that there is a unique element of $\mathbb{V}(ML_0)^{\mathrm{non}}$ lying over both $v_M$ and $w_0$.
\end{quotation}

Indeed, suppose that there is no such an element $v_M\in\mathbb{V}(M)^{\mathrm{non}}$. For each $v'\in\mathbb{V}(M)^{\mathrm{non}}$ lying over $v_0$, if the ramification index of $v'$ over $v_0$ is not divided by $l$, then there are at least two elements of $\mathbb{V}(ML_0)^{\mathrm{non}}$ lying over both $v'$ and $w_0$. Then the residue class degree of $v'$ over $v_0$ is divided by $l$. Thus, it follows from the assumption of the nonexistence of $v_M$, for each $v'\in\mathbb{V}(M)^{\mathrm{non}}$ lying over $v_0$, at least one of the ramification index and the residue class degree of $v'$ over $v_0$ is divided by $l$. Thus, the fundamental equation implies that $[M:K_0]$ is divided by $l$, which contradicts our choice of $K$. This completes the proof of Claim \hyperlink{qnftchebotarevclaim}{\ref*{qnftchebotarev}.A}.

It follows from Claim \hyperlink{qnftchebotarevclaim}{\ref*{qnftchebotarev}.A} that there exists an element $(v_M)_M\in\prod_M\mathbb{V}(M)^{\mathrm{non}}$, where $M$ runs over all finite subextensions of $K/K_0$, such that $v_M$ satisfies the condition of Claim \hyperlink{qnftchebotarevclaim}{\ref*{qnftchebotarev}.A}, and that if $M_1\subseteq M_2$, then $v_{M_2}$ is lying over $v_{M_1}$. Thus, by taking a limit, we obtain an element $v\in\mathbb{V}(K)^{\mathrm{non}}$ such that the order of the inertia subgroup of $v$ over $v_0$ is not divided by $l$, and that there exists a unique nonarchimedean place $w$ of $L=KL_0$ lying over both $v$ and $w_0$. In particular, it holds that $D_{L/K,w}=f^{-1}(D_{L_0/K_0,w_0})=f^{-1}(f(I))=I$. This completes the proof of Proposition \ref{qnftchebotarev}.
\end{proof}

\begin{lem}\label{l-inertia}
Let $p,l$ be distinct prime numbers; $F$ an algebraic extension of $\mathbb{Q}_p$ such that $[F:\mathbb{Q}_p]$ is not divided by $l^\infty$. Then the cardinality of the torsion subgroup of $(G_F^l)^{\mathrm{ab}}$ coincides with the cardinality of the set of all $l$-power roots of unity of $F$.
\end{lem}

\begin{proof}
Write $k$ for the residue field of $F$; $P$ for the wild inertia subgroup of $G_F$. Then it follows from Hensel's lemma that the cardinality of the set of all $l$-power roots of unity of $F$ coincides that of $k$. On the other hand, the $l$-adic cyclotomic character $G_k\to\mathbb{Z}_l(1)$ determines an isomorphism $(G_F/P)^l\overset{\sim}{\to}\mathbb{Z}_l(1)\rtimes G_k^l$ (respectively, $(G_F/P)^l\overset{\sim}{\to} G_k^l$) if $\zeta_l\in F$ (respectively, $\zeta_l\notin F$).

Since $P$ is pro-$p$ group, it holds that $G_F^l\cong(G_F/P)^l$. Thus, if $\zeta_l\notin F$, it holds that $(G_F^l)^{\mathrm{ab}}\cong(G_k^l)^\mathrm{ab}\cong\mathbb{Z}_l$. In particular, both the torsion subgroup of $(G_F^l)^{\mathrm{ab}}$ and the set of all $l$-power roots of unity of $F$ are of cardinality $1$. On the other hand, if $\zeta_l\in F$, then it holds that
\begin{equation*}
(G_F^l)^{\mathrm{ab}}\cong(\mathbb{Z}_l(1)\rtimes G_k^l)^{\mathrm{ab}}\cong(\mathbb{Z}_l(1))_{G_k^l}\times G_k^l\cong(\mathbb{Z}_l(1))_{G_k^l}\times\mathbb{Z}_l.
\end{equation*}
Since the coinvariant $(\mathbb{Z}_l(1))_{G_k^l}$ is naturally isomorphic to the group of all $l$-power roots of unity of $k$, we conclude that the cardinality of the torsion subgroup of $(G_F^l)^{\mathrm{ab}}$ coincides with the cardinality of the set of all $l$-power roots of unity of $k$. This completes the proof of Lemma \ref{l-inertia}.
\end{proof}

\begin{prop}\label{cyclquot-l-qnf}
Let $q,l$ be [not necessarily distinct] prime numbers; $K_1$ a $q$-quasi-number field; $K_2$ an $l$-quasi-number field; $\varphi:G_{K_1}\overset{\sim}{\to}G_{K_2}$ a continuous isomorphism; $I\subseteq G_{K_1}$ a pro-$l$ pro-cyclic closed subgroup. Then the following hold:
\begin{enumerate}
\item Let $L_1\supseteq K_1$ be a finite Galois extension of $K_1$. Write $L_2\supseteq K_2$ for the [finite Galois] extension of $K_2$ corresponding to $\varphi(G_{L_1})\subseteq G_{K_2}$; $p_1:G_{K_1}\twoheadrightarrow\Gal(L_1/K_1), p_2:G_{K_2}\twoheadrightarrow\Gal(L_2/K_2)$ for the narutal surjections. Then there exist a pro-cyclic closed subgroup $J\subseteq G_{K_1}$ and an element $\sigma\in G_{\mathbb{Q}}$ satisfying the following conditions:
\begin{itemize}
\item $p_1(I)=p_1(J)$;
\item $\sigma J\sigma^{-1}\subseteq G_{K_2}$;
\item $p_2(\sigma J\sigma^{-1})\subseteq p_2(\varphi(J))$.
\end{itemize}
\item There exists $\sigma\in G_{\mathbb{Q}}$ such that $\sigma I\sigma^{-1}\subseteq\varphi(I)$.
\end{enumerate}
\end{prop}

\begin{proof}
First, we verify assertion (i). In light of Proposition \ref{qnfext}, (iii), by replacing $K_1,K_2$ by their maximal pro-prime-to-$l$ extensions, we may assume that $K_1,K_2$ do not have nontrivial prime-to-$l$ Galois extension. Write $M_2$ for the [finite] subextension of $L_2/K_2$ corresponding to $p_2(\varphi(I))$, and let $(K_0,L_0,v,w,v_0,w_0)$ be as in Proposition \ref{qnftchebotarev}, where we take ``$(K,L,I)$'' to be $(M_2,L_2,p_2(\varphi(I)))$.
 
Let $\overline{w}\in\mathbb{V}(\overline{\mathbb{Q}})^{\mathrm{non}}$ be a place lying over $w$. Then, since $D_{L_2/M_2,w}=p_2(D_{M_2,\overline{w}})$, there exists a pro-cyclic closed subgroup $J_2\subseteq D_{M_2,\overline{w}}$ such that $p_2(J_2)=p_2(\varphi(I))(=D_{L_2/M_2,w})$. Now it follows from Proposition \ref{decomp}, (iii), that there exists $\overline{u}\in\mathbb{V}(\overline{\mathbb{Q}})^{\mathrm{non}}$ such that $\varphi(D_{M_1,\overline{u}})=D_{M_2,\overline{w}}$, where $M_1$ is the [finite] subextension of $L_1/K_1$ corresponding to $\varphi^{-1}(G_{M_2})$.

Write $u\in\mathbb{V}(M_1)^{\mathrm{non}}$ for the place lying under $\overline{u}$; $p=p_{\overline{w}}(=p_{\overline{u}})$ [cf.\ Corollary \ref{decompprime}]. Let $\sigma\in G_{\mathbb{Q}}$ be an element such that $\sigma(\overline{w})=\overline{u}$. Write $D_1^0$ (respectively, $D_2^0$) for the subgroup of $D_{\mathbb{Q},\overline{u}}$ (respectively, $D_{\mathbb{Q},\overline{w}}$) such that, by identifying $D_{\mathbb{Q},\overline{u}}$ (respectively, $D_{\mathbb{Q},\overline{w}}$) with $G_{\mathbb{Q}_p}$, $D_1^0$ (respectively, $D_2^0$) corresponds to the maximal unramified extension of $\mathbb{Q}_p$ in $(M_1)_u\cap\overline{\mathbb{Q}_p}$ (respectively, $(M_2)_v\cap\overline{\mathbb{Q}_p}$). Now we claim the following:
\begin{quotation}\hypertarget{cyclquot-l-qnfA}{}
Claim \ref*{cyclquot-l-qnf}.A: It holds that $\sigma D_1^0\sigma^{-1}=D_2^0$.
\end{quotation}

Indeed, since any unramified extension of $\mathbb{Q}_p$ is generated by [possibly infinitely many] elements of the form $\zeta_{p^n-1}$, it suffices to show that, for each positive integer $n$, $\zeta_{p^n-1}$ is contained in $(M_1)_u\cap\overline{\mathbb{Q}_p}$ if and only if it is contained in $(M_2)_v\cap\overline{\mathbb{Q}_p}$. Now, since $\varphi(D_{M_1,\overline{u}})=D_{M_2,\overline{w}}$, it follows from Lemma \ref{l-inertia} that, for each $l$-power roots of unity, it is contained in $(M_1)_u\cap\overline{\mathbb{Q}_p}$ if and only if it is contained in $(M_2)_v\cap\overline{\mathbb{Q}_p}$. On the other hand, since $K_1$ does not have nontrivial prime-to-$l$ Galois extension, by considering the residue class degree, we observe that $\zeta_{p^n-1}$ is contained in $(M_1)_u\cap\overline{\mathbb{Q}_p}$ if and only if there exists an $l$-power root of unity $\zeta$ contained in $(M_1)_u\cap\overline{\mathbb{Q}_p}$ such that $v_l(n)\leqq v_l([\mathbb{Q}_p(\zeta):\mathbb{Q}_p])$, where $v_l$ is a fixed $l$-adic valuation. Since the same holds for $(M_2)_v\cap\overline{\mathbb{Q}_p}$, we conclude that the desired equivalence holds. This completes the proof of Claim \hyperlink{cyclquot-l-qnfA}{\ref*{cyclquot-l-qnf}.A}.

Now we write $J_1\overset{\mathrm{def}}{=}\varphi^{-1}(J_2)(\subseteq\varphi^{-1}(D_{M_2,\overline{w}})=D_{M_1,\overline{u}}$); $J\overset{\mathrm{def}}{=}J_1\cap\sigma^{-1}D_{M_2,\overline{w}}\sigma$. Since $w_0$ is unramified over $\mathbb{Q}_p$ and the order of the inertia subgroup of $v$ over $v_0$ is not divided by $l$, $[D_2^0:D_{M_2,\overline{w}}]$ is not divided by $l$. On the other hand, it follows from Claim \hyperlink{cyclquot-l-qnfA}{\ref*{cyclquot-l-qnf}.A} that $\sigma J_1\sigma^{-1}\subseteq D_2^0$. Thus,
\begin{equation*}
[J_1:J]=[\sigma J_1\sigma^{-1}:\sigma J_1\sigma^{-1}\cap D_{M_2,\overline{w}}],
\end{equation*}
is not divided by $l$, which implies that $[J_1:J]$ and $[L_2:M_2]=\#\Gal(L_2/M_2)$ are coprime. In particular, it holds that $p_2(\varphi(J))=p_2(\varphi(J_1))=p_2(J_2)=p_2(\varphi(I))=D_{L_2/M_2,w}$. This implies that $p_1(I)=p_1(J)$. Moreover, it holds that $\sigma J\sigma^{-1}\subseteq D_{M_2,\overline{w}}\subseteq G_{M_2}\subseteq G_{K_2}$. Furthermore, we obtain that $p_2(\sigma J\sigma^{-1})\subseteq p_2(D_{M_2,\overline{w}})=D_{L_2/M_2,w}=p_2(\varphi(J))$. This completes the proof of assertion (i).

Next, we verify assertion (ii). For each normal open subgroup $U$ of $G_{\mathbb{Q}}$, if we write $p_1:G_{K_1}\twoheadrightarrow G_{K_1}/(G_{K_1}\cap U)$; $p_2:G_{K_2}\twoheadrightarrow G_{K_2}/(\varphi(G_{K_1}\cap U))$ for natural surjections, then it follows from assertion (i) that there exists a pro-cyclic subgroup $J\subseteq G_{K_1}$ and an element $\sigma\in G_{\mathbb{Q}}$ such that
\begin{itemize}
\item $p_1(I)=p_1(J)$;
\item $\sigma J\sigma^{-1}\subseteq G_{K_2}$;
\item $p_2(\sigma J\sigma^{-1})\subseteq p_2(\varphi(J))(=p_2(\varphi(I)))$.
\end{itemize}
Then it holds that
\begin{equation*}
\sigma I\sigma^{-1}\subseteq \sigma J(G_{K_1}\cap U)\sigma^{-1}\subseteq\sigma J\sigma^{-1}U\subseteq\varphi(\tilde{I})(\varphi(G_{K_1}\cap U))U.
\end{equation*}
Let us write
\begin{equation*}
C_U\overset{\mathrm{def}}{=}\{\sigma\in G_{\mathbb{Q}}\,\vert\,\sigma I\sigma^{-1}\subseteq\varphi(I)(\varphi(G_{K_1}\cap U))U\}.
\end{equation*}
Then $C_U$ is a nonempty [closed] subset of $G_{\mathbb{Q}}$. Moreover, for normal open subgroups $U'\subseteq U\subseteq G_{\mathbb{Q}}$ of $G_{\mathbb{Q}}$, it is clear that $C_{U'}\subseteq C_U$. Thus, since $G_{\mathbb{Q}}$ is compact, $\bigcap_U C_U\neq\emptyset$. Let $\sigma\in\bigcap_U C_U$. Then it follows from \cite{RZ}, Proposition 2.1.4, (i), that
\begin{equation*}
\sigma I\sigma^{-1}\subseteq\bigcap_U \varphi(I)(\varphi(G_{K_1}\cap U))U=\varphi(I).
\end{equation*}
This completes the proof of assertion (ii), hence also of Proposition \ref{cyclquot-l-qnf}.
\end{proof}

\begin{lem}\label{l-fp}
Let $l$ be a prime number; $K_1,K_2$ $l$-quasi-number fields; $\varphi:G_{K_1}\overset{\sim}{\to}G_{K_2}$ a continuous isomorphism; $I\subseteq G_{K_1}$ a pro-$l$ pro-cyclic closed subgroup. Then there exists $\sigma\in G_{\mathbb{Q}}$ such that $\varphi(I)=\sigma I\sigma^{-1}$.
\end{lem}

\begin{proof}
It follows from Proposition \ref{cyclquot-l-qnf}, (ii), that there exists $\sigma\in G_{\mathbb{Q}}$ such that $\sigma I\sigma^{-1}\subseteq\varphi(I)$. Moreover, by applying Proposition \ref{cyclquot-l-qnf}, (ii), for $\varphi^{-1}$, there exists $\tau\in G_{\mathbb{Q}}$ such that $\tau\varphi(I)\tau^{-1}\subseteq\varphi^{-1}(\varphi(I))=I$. Thus, it holds that $\sigma I\sigma^{-1}\subseteq\varphi(I)\subseteq\tau^{-1}I\tau$. On the other hand, by considering all finite quotients of $G_{\mathbb{Q}}$, $\sigma I\sigma^{-1}\subseteq\tau^{-1}I\tau$ implies that $\sigma I\sigma^{-1}=\tau^{-1}I\tau$. Thus, we conclude that $\varphi(I)=\sigma I\sigma^{-1}$. This completes the proof of Lemma \ref{l-fp}.
\end{proof}

\begin{lem}\label{normalpreserve}
Let $l$ be a prime number; $K_1,K_2$ $l$-quasi-number fields; $\varphi:G_{K_1}\overset{\sim}{\to}G_{K_2}$ a continuous isomorphism. Suppose that $K_1$ is Galois over $\mathbb{Q}$. Then there exists a closed subgroup $N\subseteq G_{K_1}\cap G_{K_2}$ satisfying the following conditions:
\begin{itemize}
\item $N$ is normal in $G_{\mathbb{Q}}$;
\item the [Galois] extension of $\mathbb{Q}$ corresponding to $N$ is an $l$-quasi-number field;
\item $\varphi(N)=N$.
\end{itemize}
\end{lem}

\begin{proof}
Write $G_1\subseteq G_{K_1}$ for the closed subgroup generated by all Sylow $l$-subgroups of $G_{K_1}$ [note that $G_1$ is the kernel of the maximal pro-prime-to-$l$ quotient of $G_{K_1}$, hence characteristic in $G_{K_1}$]; $G_2\overset{\mathrm{def}}{=}\varphi(G_1)$; $G_2'$ for the maximal normal closed subgroup of $G_{\mathbb{Q}}$ contained in $G_2$; $G_3\subseteq G_2'$ for the closed subgroup generated by all Sylow $l$-subgroups of $G_2'$; $G_4\overset{\mathrm{def}}{=}\varphi^{-1}(G_3)$. We verify that $G_3$ satisfies the conditions for ``$N$'' in the statement of Lemma \ref{normalpreserve}. Since $G_2'$ is normal in $G_{\mathbb{Q}}$ and $G_3$ is characteristic in $G_2'$, $G_3$ is normal in $G_{\mathbb{Q}}$. Moreover, let us observe that the separable extensions of $K_1$ (respectively, $K_2$) corresponding to $G_1,G_4\subseteq G_{K_1}$ (respectively, $G_2,G_3\subseteq G_{K_2}$) are $l$-quasi-number fields [cf.\ Proposition \ref{qnfext}, (iii)]. Thus, it suffices to show that $G_3=G_4$.

Now we claim the following:
\begin{quotation}\hypertarget{normalpreserveA}{}
Claim \ref*{normalpreserve}.A: It holds that $G_2\subseteq G_1$.
\end{quotation}

Indeed, let $S\subseteq G_2$ be a Sylow $l$-subgroup of $G_2$; $g\in S$ an element of $S$. Then it follows from Lemma \ref{l-fp} that there exists $\sigma\in G_{\mathbb{Q}}$ such that $\sigma g\sigma^{-1}\in\varphi^{-1}(G_2)=G_1$, which implies that $g\in\sigma^{-1}G_1\sigma=G_1$. Since $G_2=\varphi(G_1)$ is generated by all Sylow $l$-subgroups of $G_2$, it holds that $G_2\subseteq G_1$. This completes the proof of Claim \hyperlink{normalpreserveA}{\ref*{normalpreserve}.A}.

By applying Claim \hyperlink{normalpreserveA}{\ref*{normalpreserve}.A} to $\varphi^{-1}\vert_{G_3}$, it holds that $G_4\subseteq G_3$. Since $G_3$ is generated by all Sylow $l$-subgroups of $G_3$, to verify $G_3=G_4$, it suffices to show that, for any Sylow $l$-subgroup $S_3\subseteq G_3$ of $G_3$, it holds that $S_3\subseteq G_4$.

Since $S_3\subseteq G_3\subseteq G_2\subseteq G_1$, it follows from \cite{RZ}, Corollary 2.3.6, (b), that there exists a Sylow $l$-subgroup $S_2\subseteq G_2$ of $G_2$ containing $S_3$, and there exists a Sylow $l$-subgroup $S_1\subseteq G_1$ of $G_1$ containing $S_2$. Then, for $i=2,3$, it holds that $S_i\subseteq S_1\cap G_i\subseteq G_i$, which implies that $S_i=S_1\cap G_i$ since $S_i$ is a Sylow $l$-subgroup of $G_i$.

On the other hand, $\varphi(S_1)\subseteq G_2$ is a Sylow $l$-subgroup of $G_2$. Thus, it follows from \cite{RZ}, Corollary 2.3.6, (c), that there exists $\tau\in G_2$ such that $\varphi(S_1)=\tau S_2\tau^{-1}$, which implies that
\begin{align*}
[S_1:S_1\cap G_4]&=[\varphi(S_1):\varphi(S_1\cap G_4)]\\
&=[\tau S_2\tau^{-1}:\tau S_2\tau^{-1}\cap G_3]\\
&=[S_2:S_2\cap G_3]\\
&=[S_1\cap G_2:S_1\cap G_3].
\end{align*}
Since $G_4$ corresponds to an $l$-quasi-number field, the index $[S_1:S_1\cap G_4]$ between pro-$l$ groups is finite. Thus, $G_4\subseteq G_3\subseteq G_2\subseteq G_1$ implies that $S_1\cap G_4=S_1\cap G_3=S_3$, which implies that $S_3\subseteq G_4$. This completes the proof of Lemma \ref{normalpreserve}.
\end{proof}

\begin{thm}\label{lqnfinn}
Let $l$ be a prime number; $K_1,K_2$ $l$-quasi-number fields; $\varphi:G_{K_1}\overset{\sim}{\to}G_{K_2}$ a continuous isomorphism. Then $\varphi$ is induced by an inner automorphism of $G_{\mathbb{Q}}$.
\end{thm}

\begin{proof}
Since $G_{\mathbb{Q}}$ is internally indecomposable [cf.\ Theorem \ref{sneanab}], in light of Proposition \ref{snisom} and Corollary \ref{qnfextcor}, we may assume that $K_1$ is Galois over $\mathbb{Q}$. For a normal open subgroup $V\subseteq G_{\mathbb{Q}}$ of $G_{\mathbb{Q}}$, since $G_{K_1}\cap V$ is normal in $G_{\mathbb{Q}}$, if we write $K_V$ for the separable extension of $K_1$ corresponding to $G_{K_1}\cap V\subseteq G_{K_1}$, then $K_V$ is Galois over $\mathbb{Q}$. Write $N_V\subseteq\varphi(G_{K_1}\cap V)$ for ``$N$'' in Lemma \ref{normalpreserve}, where we take ``$K_1$'' to be $K_V$. Then it holds that $N_V\subseteq G_{K_1}$ and $\varphi(N_V)=N_V$.

Let $S\subseteq N_V$ be a Sylow $l$-subgroup of $N_V$. Then it follows from Lemma \ref{l-fp} that for any [pro-$l$] pro-cyclic closed subgroup $I\subseteq S(\subseteq G_{K_1})$, there exists $\sigma\in G_{\mathbb{Q}}$ such that $\varphi(I)=\sigma I\sigma^{-1}$. On the other hand, the image of $S\subseteq N_V$ via the natural surjective homomorphism $N_V\twoheadrightarrow N_V^l$ is a Sylow $l$-subgroup of $N_V^l$, that is, the image coincides $N_V^l$. Thus, the automorphism of $N_V^l$ determined by $\varphi$ is families preserving in $G_\mathbb{Q}/\ker(N_V\twoheadrightarrow N_V^l)$.

Now it follows from Proposition \ref{qnfext}, (iv), that the normal closed subgroup $N_V^l\subseteq G_\mathbb{Q}/\ker(N_V\twoheadrightarrow N_V^l)$ is nontrivial. Thus, it follows from Proposition \ref{snisom}; Proposition \ref{nffree}, (ii), that the automorphism $N_V^l\overset{\sim}{\to}N_V^l$, hence also the isomorphism $G_{K_1}/\ker(N_V\twoheadrightarrow N_V^l)\overset{\sim}{\to}G_{K_2}/\ker(N_V\twoheadrightarrow N_V^l)$, is induced by an inner automorphism of $G_\mathbb{Q}/\ker(N_V\twoheadrightarrow N_V^l)$. Finally, since $\bigcap_V\ker(N_V\twoheadrightarrow N_V^l)\subseteq\bigcap_V N_V\subseteq\bigcap_V(G_{K_1}\cap V)=\{1\}$, it follows from Lemma \ref{innlim} that $\varphi$ is induced by an inner automorphism of $G_{\mathbb{Q}}$. This completes the proof of Theorem \ref{lqnfinn}.
\end{proof}

\begin{rem}\label{lqnfinnrem}
\cite{O1} states that the following holds.

\begin{quotation}
Write $\mathcal{F}_1$ for the set of algebraic extensions of $\mathbb{Q}$ such that $K\in\mathcal{F}_1$ if and only if $K$ is a finite extension of a [not necessarily finite] Galois extension of $\mathbb{Q}$ that satisfies the following: for every positive integer $N$, there exists a prime number $l$ such that $N$ divides $l-1$ and $l$ does not divide $[K:\mathbb{Q}]$.

Let $K_1,K_2\in\mathcal{F}_1$; $\varphi:G_{K_1}\overset{\sim}{\to}G_{K_2}$ a continuous isomorphism. Then $\varphi$ is induced by an inner automorphism of $G_{\mathbb{Q}}$.
\end{quotation}

If $K\in\mathcal{F}_1$, then there exists an arbitrarily large prime number that does not divide $[K:\mathbb{Q}]$, thus $K$ is a $\Sigma_K$-quasi-number field for an infinite set of prime numbers $\Sigma_K$. However, for $K_1,K_2\in\mathcal{F}_1$, it is not always possible to take a common prime number from $\Sigma_{K_1}$ and $\Sigma_{K_2}$. Thus, Theorem \ref{lqnfinn} is a partial generalization of the above result in \cite{O1}. It is natural to pose the following question:
\begin{quotation}
\noindent Question: If we replace the assumption that $K_2$ is an $l$-quasi-number field with the assumption that $K_2$ is a $p$-quasi-number field for some prime number $p$, is Theorem \ref{lqnfinn} still valid?
\end{quotation}
However, at the time of writing of the present paper, the authors do not know whether this question is affirmative or not.
\end{rem}

\begin{cor}\label{lqnfbij}
Let $l$ be a prime number; $K_1,K_2$ $l$-quasi-number fields. Then the natural map
\begin{equation*}
\Isom(K_2,K_1)\to\OutIsom(G_{K_1},G_{K_2})
\end{equation*}
is bijective. In particular, $K_1$ and $K_2$ are isomorphic if and only if $G_{K_1}$ and $G_{K_2}$ are isomorphic [as profinite groups].
\end{cor}

\begin{proof}
The surjectivity immediately follows from Theorem \ref{lqnfinn}. We verify the injectivity. Write $L$ for the Galois closure of $K_2$ over $\mathbb{Q}$. Since $K_2$ is an $l$-quasi-number field, $G_L$ is a nontrivial closed subgroup of $G_{K_2}$ that is normal in $G_\mathbb{Q}$. In particular, it follows from Theorem \ref{sneanab} that $Z_{G_{\mathbb{Q}}}(G_{K_2})\subseteq Z_{G_{\mathbb{Q}}}(G_L)=\{1\}$.

Now let $f,g\in\Isom(K_2,K_1)$ be elements that map to the same element under the natural map under consideration. Let us fix an element $\sigma\in G_{\mathbb{Q}}$ such that $\sigma\vert_{K_2}=g^{-1}\circ f$. Then, by our choice of $f,g$, the conjugation by $\sigma$ determines an automorphism of $G_{K_2}$, and, moreover, this automorphism is an inner automorphism of $G_{K_2}$. In particular, there exists an element $\tau\in G_{K_2}$ such that $\tau^{-1}\sigma\in Z_{G_{\mathbb{Q}}}(G_{K_2})=\{1\}$, which implies that $g^{-1}\circ f=\sigma\vert_{K_2}=\tau\vert_{K_2}=\id_{K_2}$. This completes the proof of injectivity, hence also of Corollary \ref{lqnfbij}.
\end{proof}

\begin{rem}\label{lqnfbijrem}
In \cite{H2}, an explicit functorial ``group-theoretic algorithm'' for reconstructing a number field from its absolute Galois group is given. It is natural to pose the following question:
\begin{quotation}
\noindent Question: Is there a ``group-theoretic algorithm'' for reconstructing an $l$-quasi-number field from its absolute Galois group?
\end{quotation}
However, at the time of writing of the present paper, the authors do not know whether this question is affirmative or not.
\end{rem}

\begin{cor}\label{cyclsubext}
Write $\mathbb{Q}^{\mathrm{cyc}}$ for the cyclotomic $\widehat{\mathbb{Z}}$-extension of $\mathbb{Q}$; $\mathcal{F}$ for the set of finite extension fields of proper subfields of $\mathbb{Q}^{\mathrm{cyc}}$. Let $K_1,K_2\in\mathcal{F}$; $\varphi:G_{K_1}\overset{\sim}{\to}G_{K_2}$ a continuous isomorphism. Then $\varphi$ is induced by an inner automorphism of $G_{\mathbb{Q}}$. In particular, for $K_1,K_2\in\mathcal{F}$, the natural map
\begin{equation*}
\Isom(K_2,K_1)\to\OutIsom(G_{K_1},G_{K_2})
\end{equation*}
is bijective.
\end{cor}

\begin{proof}
If $K\in\mathcal{F}$, then there exists a proper subfield $L$ of $\mathbb{Q}^{\mathrm{cyc}}$ such that $K$ is a finite extension field of $L$. Since $\Gal(L/\mathbb{Q})$ is a quotient of $\Gal(\mathbb{Q}^{\mathrm{cyc}}/\mathbb{Q})\cong\widehat{\mathbb{Z}}$ by a nontrivial [normal] closed subgroup, there exists a prime number $l$ such that $[L:\mathbb{Q}]$, hence also $[K:\mathbb{Q}]$, is not divided by $l^\infty$, i.e., $K$ is an $l$-quasi-number field. Moreover, it follows from \cite{NSW}, Theorem 7.1.8, (i); \cite{NSW}, Corollary 8.1.18, that, for each prime number $p$, $K$ is a $p$-quasi-number field if and only if $\cdim_p G_K=1$. Thus, Corollary \ref{cyclsubext} follows from Theorem \ref{lqnfinn}; Corollary \ref{lqnfbij}.
\end{proof}

\begin{rem}\label{cyclsubextrem}
Since $\cdim G_{\mathbb{Q}^{\mathrm{cyc}}}=1$, it cannot be expected that any finite extension field of $\mathbb{Q}^{\mathrm{cyc}}$ will satisfy anabelian properties like as Corollary \ref{cyclsubext}.
\end{rem}

\vskip.5\baselineskip
\begin{center}
\textbf{Acknowledgements}
\end{center}

The authors would like to thank Professor Akio Tamagawa for some comments. The authors would like to express their sincere gratitude to the organizers of the various workshops and conferences held during the period 2024--2026 for giving the second author numerous opportunities to present the results contained in this paper. The authors regret the considerable delay between the first public announcement of these results and the public release of the present paper, during which time the results were strengthened and the exposition was substantially revised.

The first author was supported by JSPS KAKENHI Grant Number JP24K16898. The second author was supported by JSPS KAKENHI Grant Number JP20J00323. This research was supported by the Research Institute for Mathematical Sciences, an International Joint Usage/Research Center located in Kyoto University, as well as the Center for Research in Next Generation Geometry. This work is part of the ``Arithmetic and Homotopic Galois Theory'' project, supported by the CNRS France-Japan AHGT International Research Network between the RIMS Kyoto University, the LPP of Lille University, and the DMA of ENS PSL.

\vskip.5\baselineskip

\end{document}